\documentclass[a4paper,11pt]{article}       
\usepackage[utf8]{inputenc}     
\usepackage{enumerate}          
\usepackage{amssymb,amsmath,amsfonts,amsxtra,amsthm}            
\usepackage{pdfsync}
\usepackage{xcolor}
\usepackage{url}
\usepackage{mathrsfs}
\usepackage{mathtools}
\usepackage{hyperref}
\usepackage{anysize}
\usepackage{ulem}
\usepackage{comment}

\begin{document}
\theoremstyle{plain}
\newtheorem{definition}{Definition}
\newtheorem{theorem}{Theorem}
\newtheorem{proposition}[definition]{Proposition}
\newtheorem{lemma}[definition]{Lemma}
\newtheorem{corollary}[definition]{Corollary}
\newtheorem{example}[definition]{Example}
\newtheorem{conjecture}{Conjecture}
\newtheorem{problem}{Problem}
\newtheorem{assumption}[definition]{Assumption}
\newtheorem*{discussion}{Discussion}
\newtheorem{question}[equation]{Question}
\newtheorem{remark}[definition]{Remark}
\newtheorem*{notations}{Notations}

\numberwithin{equation}{section}
\numberwithin{definition}{section}
\errorcontextlines=0

\newcommand{\C}{\mathbb{C}}
\newcommand{\R}{\mathbb{R}}
\newcommand{\N}{\mathbb{N}}
\newcommand{\E}{\mathbb{E}}
\newcommand{\Z}{\mathbb{Z}}
\newcommand{\T}{\mathbb{T}}
\newcommand{\Ensp}{\mathcal{D}_{\mathcal{E}}(\Omega)}
\renewcommand{\o}{\text{o}}
\newcommand{\ord}{\text{ord}}
\renewcommand{\d}{\partial}
\newcommand{\supp}{\text{Supp}}
\newcommand{\phg}{\text{phg}}
\newcommand{\dom}{\mathcal{D}(\Delta_\Omega)}
\newcommand{\domM}{\mathcal{D}(\Delta_M)}
\newcommand{\domq}{\mathcal{Q}(\Omega)}
\newcommand{\domqhat}{\mathcal{D}(\widehat{q}_\Omega)}
\newcommand{\domqMhat}{\mathcal{D}(\widehat{q}_M)}
\newcommand{\domqDhat}{\mathcal{D}(\widehat{q}_D)}
\newcommand{\domU}{\mathcal{D}(\Delta_U)}
\renewcommand{\hom}{\text{hom}}
\newcommand{\sR}{\text{sR}}
\newcommand{\Op}{\text{Op}}
\newcommand{\J}{\mathcal{J}}
\newcommand{\e}{\varepsilon}
\renewcommand{\leq}{\leqslant}
\newcommand{\dist}{{\rm dist}}
\renewcommand{\Box}{{\rm Box}}
\renewcommand{\geq}{\geqslant}
\renewcommand{\div}{\text{div}}
\newcommand{\todo}[1]{$\clubsuit$ {\tt #1}}
\newcommand{\mult}{\text{mult}}

\title{Nodal sets of eigenfunctions of sub-Laplacians}
\date{\today}
\author{Suresh ESWARATHASAN\footnote{Dalhousie University, Halifax NS B3H 4R2 (\texttt{sr766936@dal.ca})} \ and Cyril LETROUIT\footnote{Massachussets Institute of Technology, Cambridge MA 02139 (\texttt{letrouit@mit.edu})}}
\maketitle
\begin{abstract}
Nodal sets of eigenfunctions of elliptic operators on compact manifolds have been studied extensively over the past decades. In this note, we initiate the study of nodal sets of eigenfunctions of hypoelliptic operators on compact manifolds, focusing on sub-Laplacians. A standard example is the sum of squares of bracket-generating vector fields on compact quotients of the Heisenberg group.  Our results show that nodal sets behave in an anisotropic way which can be analyzed with standard tools from sub-Riemannian geometry such as sub-Riemannian dilations, nilpotent approximation and desingularization at singular points. Furthermore, we provide a simple example demonstrating that for sub-Laplacians, the Hausdorff measure of nodal sets of eigenfunctions cannot be bounded above by $\sqrt\lambda$, which is the bound conjectured by Yau for Laplace-Beltrami operators on smooth manifolds.
\end{abstract}

%2h50-
\section{Introduction and main results}
\subsection{Eigenfunctions and nodal sets}
Let $\Omega$ be a bounded open subset of $\R^N$ and let $\Delta$ be the (non-positive) Laplacian on $\Omega$ with Dirichlet  boundary conditions on $\partial\Omega$. It has a compact resolvent, and the spectrum of $-\Delta$, denoted by
$$
0\leq \lambda_1\leq \lambda_2\leq \ldots\leq \lambda_n\leq \ldots\rightarrow +\infty,
$$
 is discrete and tends to $+\infty$. The nodal set $Z_{\varphi_\lambda}=\varphi_\lambda^{-1}(0)$ of a mode $\varphi_\lambda$ satisfying $-\Delta \varphi_\lambda=\lambda \varphi_\lambda$ is physically interpreted as the set of nodes of the vibration profile $\varphi_\lambda$ of a drum of shape $\Omega$, and it can be observed, when $N=2$, by pouring sand on a vibrating drum of shape $\Omega$ (the famous Chladni experiment). 

On the mathematical side, nodal sets of eigenfunctions on Euclidean domains and Riemannian manifolds have been studied extensively over the past decades. Let us only mention Courant's nodal domain theorem and Pleijel's asymptotic bound on the number of nodal domains, Donnelly and Fefferman's proof of Yau's conjecture in the analytic setting, and the recent advances by Logunov and Malinnikova concerning the smooth case of the Yau conjecture. For all these results and a global overview of the subject we refer the reader to \cite{log} and the references therein. A stream of results has also emerged in the past twenty years regarding nodal sets of random waves (i.e., random linear combinations of eigenfunctions); for this we refer the reader to the recent survey \cite{wig} and the references therein.

In the present paper, we initiate the study of nodal sets of eigenfunctions for operators more general than Laplacians, specifically a certain class of  \textit{non-elliptic} operators called sub-Laplacians. Sub-Riemannian geometry is a natural extension of Riemannian geometry whose tools are suited to the analysis of sub-Laplacians.

\subsection{Sub-Laplacians}\label{s:subLap}
We start with the general definition of sub-Laplacians, and give immediately after concrete examples of sub-Laplacians.

Let $N\in\mathbb{N}^*=\mathbb{N}\setminus\{0\}$ and let $M$ be either $\R^N$ or a smooth, connected, compact manifold of dimension $N$ without boundary, endowed with a smooth volume $\mu$. Let $X_1,\ldots, X_m$ be smooth vector fields on $M$ satisfying Hörmander's bracket-generating condition: 
\begin{equation}\label{e:hormander}
\begin{split}
&\text{The vector fields $X_1,\ldots,X_m$ and their iterated Lie brackets $[X_i,X_j], [X_i,[X_j,X_k]]$, etc.}\\
 &\qquad\qquad\qquad \text{span the tangent space $T_xM$ at every point $x\in M$.}
\end{split}
\end{equation}
In particular, $X_1,\ldots,X_m$ are not assumed to span $TM$. The sub-Laplacian $\Delta$ is then defined as
\begin{equation}\label{e:sLapla}
\Delta=-\sum_{i=1}^m X_i^*X_i
\end{equation}
where $X_i^*=-X_i-{\rm div}_\mu(X_i)$ is the adjoint of $X_i$ in $L^2(M,\mu)$. Sub-Laplacians are in general not elliptic since $X_1,\ldots,X_m$ are not assumed to span $TM$.

 Instead of the usual elliptic estimates, sub-Laplacians satisfy subelliptic estimates of the form\footnote{if $M=\R^N$, for any compact set $K\subset \R^N$ there exists a constant $C$ such that \eqref{e:subest} holds for any smooth function supported in $K$.} (see \cite[Inequality (3.4)]{hor}, \cite[Theorem 1.5]{lau})
\begin{equation}\label{e:subest}
\|u\|_{H^{2/r}(M)}\leq C(\|u\|_{L^2(M)}+\|\Delta u\|_{L^2(M)}),
\end{equation}
where $r\in\N^*$ is called the step of the sub-Laplacian $\Delta$ and will be defined later. We only mention that when $X_1,\ldots,X_m$ span $TM$, then $r=1$ and we recover usual elliptic estimates. In turn the estimate \eqref{e:subest} implies the following form of regularity, called hypoellipticity (see \cite{hor}): if $\Delta u\in C^\infty(U)$ for some open set $U\subset M$, then $u\in C^\infty(U)$. 

Sub-Laplacians are a natural generalization of Euclidean Laplacians and of the Laplace-Beltrami operator in Riemannian geometry.\footnote{To see that the Laplace-Beltrami operator on a Riemannian manifold $(M,g)$ is a sub-Laplacian, take $\mu$ to be the Riemannian volume and take a partition of unity $1=\sum_{k=0}^{K-1} \chi_k(x)^2$ where $\chi_k$ is smooth and supported in a chart where $g$ admits an orthonormal frame $(Y^k_1,\ldots,Y^k_N)$. Then set $m=KN$ and $X_{Nk+j}=\chi_kY^k_j$ for $j\in\{1,\ldots,N\}$ and $0\leq k\leq K-1$. One can check that with this construction $\sum_{i=1}^mX_i^*X_i=-\Delta_g$.} They have been studied extensively since the 1960's and Hörmander's seminal work \cite{hor}. Let us illustrate the definition with several examples.

\begin{example} \label{exGrushin}
On $M=(-1,1)_x\times \mathbb{T}_y$ (where $\mathbb{T}=\R/2\pi\Z$), we set $$\Delta_{BG}=\partial_x^2+x^2\partial_y^2.$$ This sub-Laplacian is the so-called Baouendi-Grushin operator. In this case, $X_1=\partial_x$, $X_2=x\partial_y$ and $\mu$ is the Lebesgue measure $dxdy$. Then $\Delta_{BG}$ is elliptic except along the singular line $x=0$, where $X_2=0$. Along this line, $X_1=\partial_x$ and $[X_1,X_2]=\partial_y$ span the tangent space, so that \eqref{e:hormander} is satisfied. 
\end{example}
More generally, for any $\alpha\in\N^*$ we can consider the sub-Laplacian $\partial_x^2+x^{2\alpha}\partial_y^2$, which is also elliptic except along $\{x=0\}$. In this case $X_1=\partial_x$, $X_2=x^\alpha\partial_y$, and the bracket $[X_1,[X_1,\ldots,[X_1,X_2]\ldots]$ where $X_1$ appears $\alpha$ times generates the missing direction $\partial_y$.

\begin{example} \label{exHeis}
Sub-Laplacians can be defined on real Heisenberg groups in a natural way. Given $d\in\N^*$, we denote by $\mathbf{H}_d$ the Heisenberg group of dimension $2d+1$, i.e., $\R^{2d+1}$ endowed with the group law 
\begin{equation}\label{e:grouplaw}
(x,y,z) \cdot (x',y',z'):=(x+x', y+y',z+z'- \sum_{j=1}^dx_jy_j'),
\end{equation}
where $x,y,x',y'\in\R^d$ and $z,z'\in\R$. When $d=1$, this group law comes from the representation of $\mathbf{H}_1$ as the group of $3\times 3$ matrices of the form
$$
\begin{pmatrix}
1&x&-z\\ 0&1&y\\0&0&1
\end{pmatrix}
$$
for $x,y,z\in\R$, endowed with the usual product of matrices (the higher-dimensional case has a similar representation).

The set of left-invariant vector fields on $\mathbf{H}_d$ is generated by $\partial_z$ together with the $2d$ vector fields
\begin{equation*}
X_j=\partial_{x_j}, \ \  \ Y_j=\partial_{y_j}-x_j\partial_{z}, \qquad \text{for } j=1,\ldots,d.
\end{equation*}
The Heisenberg group $\mathbf{H}_d$ admits lattices, such as $\Gamma=(\sqrt{2\pi}\Z)^{2d}\times 2\pi\Z$. Left-invariant vector fields can be considered as vector fields on the compact left-quotient $\Gamma\backslash \mathbf{H}_d$.
We define the sub-Laplacian on $\Gamma\backslash \mathbf{H}_d$ by
\begin{equation*}
\Delta_{\Gamma \backslash \textbf{H}_d}=\sum_{j=1}^d X_j^2+Y_j^2.
\end{equation*}
We note that $[X_j,Y_j]=-\partial_z$ for any $j$, hence \eqref{e:hormander} is verified. 
\end{example}

\begin{example}\label{exCarnot}
More generally, sub-Laplacians arise naturally in the setting of Carnot groups, which are a family of nilpotent Lie groups whose simplest non-Euclidean examples are real Heisenberg groups (see Example \ref{exHeis}). A Carnot group $G$ of step $r$ is a connected, simply connected, finite-dimensional Lie group whose Lie algebra $\mathfrak{g}$ admits a step-$r$ stratification, meaning that there exist nontrivial linear subspaces $V_1,\ldots,V_r$ such that
$$
\mathfrak{g}=V_1\oplus \ldots \oplus V_r, \qquad [V_1,V_i]=V_{i+1} \text{ for } i=1,\ldots,r-1 \text{ and } [V_1,V_r]=0.
$$
The stratum $V_i$ contains iterated brackets of length $i$ between elements of $V_1$.
Taking a basis of $V_1$ composed of left-invariant vector fields $X_1,\ldots,X_m$ (where $m$ is the dimension of $V_1$), and taking $\mu$ to be the Haar measure on $G$, we obtain a sub-Laplacian thanks to the formula \eqref{e:sLapla}. Quotienting by an appropriate lattice and using the left-invariance of $\Delta$, we can define sub-Laplacians on compact quotients of $G$ (see Example \ref{exHeis}).

To make this example more concrete, we now explain how $G$ can be identified to $(\R^N,\star)$ for $N=\sum_{i=1}^r \dim(V_i)$ and $\star$ a specific group law. For each $i\in\{1,\ldots,r\}$ we pick a basis of vector fields of $V_i$, which all together form a basis $X_1,\ldots,X_N$ of $\mathfrak{g}$. Then the exponential map $\exp$ is  a globally defined diffeomorphism from $\mathfrak{g}\cong \R^N$ to $G$ (since $G$ is connected and nilpotent):
$$
\exp:(x_1,\ldots,x_N)\mapsto \exp\left(\sum_{i=1}^N x_iX_i\right)\in G.
$$
This allows to identify $G$ with $\R^N$, and the group law $\star$ derives from the Baker-Campbell-Hausdorff formula. Following this procedure with $X_1=\partial_x$, $X_2=\partial_y-x\partial_z$ and $X_3=\partial_z$, the group which is obtained is isomorphic to the Heisenberg group $\mathbf{H}_1$ of Example \ref{exHeis}. Besides Heisenberg groups, other examples of Carnot groups include the Engel group, for which $r=3$.
\end{example}

Our results concern the nodal sets of eigenfunctions of $\Delta$ (i.e., the set where an eigenfunction $\varphi$ vanishes) and the nodal components (i.e., the connected components of $\{\varphi\neq 0\}$). They extend results which are well-known for eigenfunctions of Laplace-Beltrami operators.

\subsection{Main results: Courant's theorem and density of the nodal set}\label{s:courantandnodal}
Our results address the validity of Courant's nodal domain theorem  \cite{cou} for eigenfunctions of sub-Laplacians and the $1/\sqrt{\lambda}$-density of the nodal set with respect to an adapted distance, called the sub-Riemannian distance.

Let $\Omega\subset M$ be a connected open subset, assumed to be bounded and with Lipschitz boundary if $M=\R^N$. We denote by $C_c^\infty(\Omega)$ the set of smooth functions whose support is contained in $\Omega$. The operator $\Delta:C_c^\infty(\Omega)\rightarrow C_c^\infty(\Omega)$ is non-positive, symmetric and densely defined in $L^2(\Omega,\mu)$. In the sequel, we denote by $(\Delta_\Omega,\dom)$ its Friedrichs extension (see \cite{Riesz} and Section \ref{s:proofprop} for reminders). When $\partial\Omega\neq \emptyset$, this naturally enforces Dirichlet boundary conditions.

\begin{proposition} \label{p:eigpb}
The selfadjoint operator $(-\Delta_\Omega,\dom)$ has discrete point spectrum $$0\leq\lambda_1\leq \lambda_2\leq \ldots \leq \lambda_n \leq \ldots\rightarrow +\infty$$ (with repetitions according to multiplicities). There exists an orthonormal basis $\{\varphi_n\}_{n\in\N}$ of $L^2(\Omega,\mu)$ such that for every $n\in\N$, $\varphi_n\in\dom$ and $-\Delta_\Omega\varphi_n=\lambda_n\varphi_n$. 
\end{proposition}

In our next results, we are going to rely on two assumptions.
\begin{assumption} \label{a:noncarac}
We assume that either $\Omega=M$, or the boundary of $\Omega$ is smooth and non-characteristic, meaning that for any $x\in \partial\Omega$, there exists $i\in\{1,\ldots,m\}$ such that $X_i(x)\notin T_x\partial\Omega$.
\end{assumption} 
Under Assumption \ref{a:noncarac}, the eigenfunctions of $-\Delta_\Omega$ are smooth up to the boundary $\partial \Omega$ (see \cite[Theorem III, point (4)]{kohnnirenberg}), which will be important in the proofs.
\begin{assumption} \label{a:unique}
We assume that one of the following holds:
\begin{enumerate}
\item The (topological) dimension of $M$ is $N=2$;
\item or the manifold $M$, the volume $\mu$ and the vector fields $X_1,\ldots,X_m$ are real-analytic.
\end{enumerate}
\end{assumption}
Under this assumption, it follows from \cite[Theorem 1]{wat} (in case 1)  and \cite{bony} (in case 2) that any $u$ satisfying $(\Delta_\Omega+\lambda) u=0$ in $\Omega$ for some $\lambda\in \R$ and vanishing in a non-empty open subset $U\subset \Omega$ vanishes in fact everywhere in $\Omega$.  Actually, to prove this unique continuation property in case 2, instead of working directly with the operator $\Delta_\Omega+\lambda$ which does not satisfy the assumptions of \cite{bony}, we work with $\partial_t^2+\Delta_\Omega$ instead.  We write that $v=ue^{i\lambda t}$ is a solution of the equation $(\partial_t^2+\Delta_\Omega)v=0$ in $\Omega\times \R$, which vanishes in the non-empty open subset $U\times \R$. Hence $v$ vanishes everywhere in $\Omega\times \R$ according to \cite[Theorem 1]{wat} (in case 1) and \cite[Corollary 4.1]{bony} (in case 2), and consequently $u$ vanishes everywhere in $\Omega$. This unique continuation property will be used in the proof of the second part of Theorem \ref{t:courant}. 

We denote by $$Z_f=\{f=0\}\subset \overline{\Omega}$$ the zero set, or commonly the \textit{nodal set} of a function $f$. A \textit{nodal domain} of $f$ is a connected component of $\Omega \setminus Z_f$.  We are in position to state an analogue of the Courant nodal domain theorem for sub-Laplacians:
\begin{theorem} \label{t:courant}
Under Assumption \ref{a:noncarac}, for any $n\in\N$, any eigenfunction of $-\Delta_\Omega$ with eigenvalue $\lambda_n$ has at most $n+\mult(\lambda_n)-1$ nodal domains, where $\mult(\lambda_n)$ denotes the multiplicity of $\lambda_n$. If moreover Assumption \ref{a:unique} is satisfied, a stronger bound holds: the number of nodal domains of any eigenfunction with eigenvalue $\lambda_n$ is bounded above by $n$.
\end{theorem}

For our second result we introduce the sub-Riemannian metric: it is defined for $q\in M$ and $v\in T_qM$ as 
\begin{equation}\label{e:me}
g_q(v)=\inf\left\{\sum_{i=1}^mu_i^2 \ | \ v=\sum_{i=1}^mu_iX_i\right\}.
\end{equation}
This metric is finite if and only if $v\in\mathcal{D}={\rm Span}(X_1,\ldots,X_m)$. It induces a notion of distance $d:M\times M\rightarrow \R_+$, and the distance between two points is always finite thanks to the bracket-generating condition \eqref{e:hormander} (due to the Chow-Rashevsky theorem \cite[Theorem 2.4]{bellaiche1996tangent}). The sub-Riemannian balls are then defined as
\begin{equation} \label{e:ball}
B_\e(q)=\left\{q'\in M,\ d(q,q')<\e\right\}
\end{equation}
for $q\in M$ and $\e>0$.
\begin{theorem} \label{t:density}
Under Assumption \ref{a:noncarac}, there exists $C>0$ depending only on $\Omega$ such that for any $\lambda\in\R$ and any eigenfunction $\varphi_\lambda$ with eigenvalue $\lambda$, the nodal set $Z_{\varphi_\lambda}$ intersects any sub-Riemannian ball of radius greater than $C\lambda^{-1/2}$.
\end{theorem}
 Theorem \ref{t:density} is illustrated by a simple example in Section \ref{s:optimal} below. 

An important distinguishing factor of sub-Riemannian geometry from Riemannian geometry is that sub-Riemannian balls are \textit{anisotropic} whilst Riemannian balls are isotropic. The ball-box theorem captures this difference in more precise language (see \cite[Corollary 7.35]{bellaiche1996tangent}): it states that at any point $q\in\Omega$, there exist a system of privileged coordinates $\psi_q=(x_1,\ldots,x_N):U\rightarrow\R^N$ defined in a neighborhood $U$ of $q$, positive integers $w_1,\ldots,w_N$, and constants $C_q,\e_q>0$ such that for any $\e<\e_q$, the sub-Riemannian ball $B_\e(q)$ verifies
\begin{equation}\label{e:bboxthe}
\Box(C_q^{-1}\e)\subset (\psi_q)_*(B_\e(q))\subset \Box(C_q\e)
\end{equation}
where 
\begin{equation}\label{e:defbox}
\Box(\e)=[-\e^{w_1},\e^{w_1}]_{x_1}\times \ldots\times [-\e^{w_N},\e^{w_N}]_{x_N}\subset \R^N
\end{equation}
(note that $C_q$ does not depend on $\e$). Sub-Riemannian balls can thus be approximated by Euclidean rectangles in an appropriate coordinate system, with side-lengths scaling differently in each direction. Section \ref{s:optimal}'s example serves as a useful illustrator of the ball-box theorem.

As a consequence, roughly speaking, if $\Omega$ is some subset of $\R^N$, starting from a fixed point $q\in\Omega$ and following a given direction $\vec{\ell}$, one expects to cross the nodal set many more times in some directions $\vec{\ell}$ than some others, when the eigenvalue is large. This phenomenon is also explicitly expressed by Section \ref{s:optimal}'s example.

\subsection{Yau-type bounds for sub-Laplacians} 
In this section we explain that Yau's conjecture, unless properly modified, is not true for sub-Laplacians. Yau conjectured that for any smooth closed Riemannian manifold $(M,g)$ of dimension $N$,  there exist constants $c,C>0$ such that for any eigenfunction $\varphi_\lambda$ of the associated Laplace-Beltrami operator $-\Delta_g$ with eigenvalue $\lambda$,
\begin{equation}\label{e:yaubound}
c\sqrt{\lambda}\leq \mathscr{H}^{N-1}(Z_{\varphi_\lambda})\leq C\sqrt{\lambda}.
\end{equation}
where $\mathscr{H}^{N-1}$ is the $(N-1)$-dimensional Hausdorff measure. This conjecture has been proved in the real-analytic setting by Donnelly and Fefferman \cite{don}, but it is open in the case of general smooth manifolds.

In sub-Riemannian geometry, it is known, under the condition that the $w_i$ in \eqref{e:defbox} are independent of $q$ \footnote{this is equivalent to the condition that the sub-Riemannian flag is equiregular, see Section \ref{s:sRflag} for definitions.}, that the Hausdorff dimension with respect to the sub-Riemannian metric \eqref{e:me} of $M$ is equal to 
\begin{equation}\label{e:premierQ}
\mathcal{Q}=\sum_{i=1}^N w_i
\end{equation}
(see \cite{mitchell85}), which is strictly greater than $N$ as soon as the vector fields $X_1,\ldots,X_m$ do not span $TM$. For instance,  in the Heisenberg group $\mathbf{H}_1$ of Example \ref{exHeis}, there holds $N=3$, $w_1=1$, $w_2=1$, $w_3=2$, and hence  $\mathcal{Q}=4$. It is also known that hypersurfaces have Hausdorff dimension $\mathcal{Q}-1$ with respect to the sub-Riemannian metric \eqref{e:me} (see \cite[Section 0.6.C]{gromov}).

The following result shows that Yau-type bounds like \eqref{e:yaubound} (with $N$ replaced by $\mathcal{Q}$) do not hold for sub-Laplacians: in general we do not expect better bounds than 
\begin{equation*}
c\sqrt{\lambda}\leq \mathscr{H}^{\mathcal{Q}-1}(Z_{\varphi_\lambda})\leq C\lambda^{r/2}
\end{equation*}
where the step $r$ is defined in Section \ref{s:sRflag} below and is equal to $2$ in the Heisenberg case. Here and in the sequel, $\mathscr{H}^K$ denotes the Hausdorff measure of dimension $K$.
%and $\mathscr{H}{\mathcal{Q}-1}$ denotes the $(\mathcal{Q}-1)$-dimensional Hausdorff measure with respect to the 
\begin{theorem}\label{t:yau}
Let $M=\Gamma\backslash \mathbf{H}_1$ as in Example \ref{exHeis}, endowed with the Lebesgue measure and the vector fields $X_1=\partial_{x_1}$ and $Y_1=\partial_{y_1}-x_1\partial_{z_1}$. Let $\mathscr{H}^3$ denote the Hausdorff measure of dimension $3$ with respect to the associated sub-Riemannian metric on $M$. There exist an open subset $\Omega\subset M$, constants $c_1,c_2>0$ and sequences of eigenfunctions $(\varphi_{1,m})_{m\in\N}$ and $(\varphi_{2,m})_{m\in\N}$ of $-\Delta_\Omega$ with respective eigenvalues $(\lambda_{1,m})_{m\in\N}$ and $(\lambda_{2,m})_{m\in\N}$ tending to $+\infty$, such that 
\begin{equation}\label{e:loweryau}
\mathscr{H}^{3}(Z_{\varphi_{1,m}})\leq c_1\sqrt{\lambda_{1,m}}
\end{equation}
and 
\begin{equation}\label{e:yauupper}
\mathscr{H}^{3}(Z_{\varphi_{2,m}})\geq c_2\lambda_{2,m}.
\end{equation}
\end{theorem}
The fact that the two bounds in our result are not of the same order is due to the fact that some sequences of eigenfunctions oscillate much more in the directions needing brackets to be generated (like $\partial_z$ in the Heisenberg group) than what is possible for Laplacians.  This demonstrates the need for a reformulation of Yau's Conjecture.

\subsection{Open questions}
\subsubsection{Pleijel bound}
The Courant bound for Laplace-Beltrami operators is known to be non-optimal as the eigenvalue tends to $+\infty$: the Pleijel bound asserts that when $n$ becomes large, the number of nodal domains of an eigenfunction with eigenvalue $\lambda_n$ is at most $cn$ for some explicit constant $c<1$ (see \cite{ple} for the case of 2-dimensional Euclidean domains, and \cite{ber}). An analogous result for sub-Laplacians would require new ideas since the appropriate Faber-Krahn inequality that one uses in the usual proof of the Pleijel bound is not known.

\subsubsection{Relation between elliptic and subelliptic bounds} In the analytic case, Yau's conjectured bound \eqref{e:yaubound} for the Laplace-Beltrami operator is known to hold (see \cite{don}). If we consider a family of Laplace-Beltrami operators of the form
\begin{equation}\label{e:blowup}
\Delta_{g_\varepsilon}=\Delta_{\rm sR}+\varepsilon^2 \Delta_{h}
\end{equation}
where $\Delta_{\rm sR}$ is a fixed analytic sub-Laplacian and $\Delta_h$ is a fixed analytic Laplace-Beltrami operator (defined on the same domain of $\R^N$), how do the constants $c$ and $C$ in \eqref{e:yaubound} behave as $\varepsilon\rightarrow 0$? This question is motivated by the fact that the ``sub-Riemannian limit'' of operators \eqref{e:blowup} has already attracted attention in other related contexts, see for instance \cite{rum00}.

\subsection{Organization of the paper and new contributions}
The proofs of our results follow the same broad strokes as those for elliptic Laplacians.  However, to the best of our knowledge, it is the novel application of tools and techniques unique to sub-Riemannian geometry and hypoelliptic equations within the subject of nodal geometry that is our main contribution.  We elaborate on this in our outline.

In Section \ref{s:proofprop}, we prove Proposition \ref{p:eigpb}: our proof mostly relies on the compactness of the resolvent of sub-Laplacians, which follows from well-known subelliptic estimates due to H\"ormander. In Section \ref{s:proofcourant} we prove Theorem \ref{t:courant} using the usual strategy for proving Courant-type bounds, and a combination of a robust argument by Colette Anné \cite{anne} with subelliptic estimates. In Section \ref{s:sRtools}, we gather tools coming from sub-Riemannian geometry, namely the nilpotent approximation and the desingularization procedure. In Section \ref{s:theorem2}, we prove Theorem \ref{t:density} using the aforementioned tools as basic building blocks for estimating the first eigenvalue of a sub-Laplacian in a small ball. Finally, in Section \ref{s:yau} we prove Theorem \ref{t:yau} by constructing a sub-Laplacian and explicit examples of eigenfunctions.

\paragraph{Acknowledgments.} We thank Hajer Bahouri, Yves Colin de Verdière, Valentina Franceschi, Bernard Helffer, Thomas Letendre, Eugenia Malinnikova, Iosif Polterovich and Luca Rizzi  for interesting discussions. The first author is supported by an NSERC Discovery Grant and the second author is supported by the Simons Foundation Grant 601948, DJ.

\section{Proof of Proposition \ref{p:eigpb}}\label{s:proofprop}
We first recall briefly the classical Friedrichs extension construction (\cite{Riesz}). We denote by $q_\Omega$ the quadratic form on $C_c^\infty(\Omega)$ given by $q_\Omega(v,w)=(\Delta v, w)$ where $(\cdot,\cdot)$ denotes the $L^2(\Omega,\mu)$ scalar product. It is closable and we denote by $\widehat{q}_\Omega$ its closure. Explicitly, denoting by $H$ the Hilbert space completion of $C_c^\infty(\Omega)$ with respect to the scalar product $(v,w)_H=(v,w)+q_\Omega(v,w)$, the inclusion map $\iota:C_c^\infty(\Omega)\xhookrightarrow{} L^2(\Omega,\mu)$ extends by continuity to a linear map $\widehat{\iota}:H\rightarrow L^2(\Omega,\mu)$. The quadratic form $q_\Omega$ also extends by continuity to a quadratic form $\widehat{q}_\Omega$ over $H$, so that if $\overline{v}$ and $\overline{w}$ denote the equivalence classes of $\{v_n\}$ and $\{w_n\}$ in $H$, then $\widehat{q}(\overline{v},\overline{w})=\lim_{n\rightarrow \infty}q_\Omega(v_n,w_n)$. One can check that $\widehat{\iota}$ is injective, hence $\widehat{q}_\Omega$ can be seen as a quadratic form on $L^2(\Omega,\mu)$, with domain $\domqhat=\widehat{\iota}(H)$. More concretely, the domain $\domqhat$ consists of those $v\in L^2(\Omega,\mu)$ such that there exists $\{v_n\}\subset C_c^\infty(\Omega)$ such that $v_n\rightarrow v$ in $L^2(\Omega,\mu)$ and $q_\Omega(v_n-v_\ell)\rightarrow 0$ as $\ell,n\rightarrow \infty$. 

%In this section, we prove Proposition \ref{p:eigpb}.
%
%We first recall briefly the classical Friedrichs extension construction. We denote by $q_\Omega$ the quadratic form on $C_c^\infty(\Omega)$ given by $q_\Omega(v,w)=(\Delta v, w)$ where $(\cdot,\cdot)$ denotes the $L^2(\Omega,\mu)$ scalar product. It is closable and we denote by $\widehat{q}_\Omega$ its closure. Explicitly, denoting by $H$ the Hilbert space completion of $C_c^\infty(\Omega)$ with respect to the scalar product $(v,w)_H=(v,w)+q_\Omega(v,w)$, the inclusion map $\iota:C_c^\infty(\Omega)\xhookrightarrow{} L^2(\Omega,\mu)$ extends by continuity to a linear map $\widehat{\iota}:H\rightarrow L^2(\Omega,\mu)$. The quadratic form $q_\Omega$ also extends by continuity to a quadratic form $\widehat{q}_\Omega$ over $H$, so that if $\overline{v}$ and $\overline{w}$ denote the equivalence classes of $\{v_n\}$ and $\{w_n\}$ in $H$, then $\widehat{q}(\overline{v},\overline{w})=\lim_{n\rightarrow \infty}q_\Omega(v_n,w_n)$. One can check that $\widehat{\iota}$ is injective, hence $\widehat{q}_\Omega$ can be seen as a quadratic form on $L^2(\Omega,\mu)$, with domain $\mathcal{D}(\widehat{q}_\Omega)=\widehat{\iota}(H)$. More concretely, the domain $\mathcal{D}(\widehat{q}_\Omega)$ consists of those $v\in L^2(\Omega,\mu)$ such that there exists $\{v_n\}\subset C_c^\infty(\Omega)$ such that $v_n\rightarrow v$ in $L^2(\Omega,\mu)$ and $q_\Omega(v_n-v_\ell)\rightarrow 0$ as $\ell,n\rightarrow \infty$. 

Then, the Friedrichs extension of $(\Delta,C_c^\infty(\Omega))$ is the operator $(\Delta_\Omega,\dom)$ where
\begin{equation}\label{e:domDelta}
 \dom=\{v\in\domqhat : \widehat{q}_\Omega(v,\cdot) \text{ is } L^2(\Omega,\mu)-\text{continuous}\}
\end{equation}
 and $\Delta_\Omega v\in L^2(\Omega,\mu)$ is defined through the Riesz representation theorem by the relation $(\Delta_\Omega v, w)=\widehat{q}_\Omega(v,w)$ for any $w\in \domqhat$ (note that $\domqhat$ is dense in $L^2(\Omega,\mu)$). By the same procedure, for any bounded open set $U\subset M$ (including the cases of $U=M$ and $U=\Omega$) we obtain the Friedrichs extension $(\Delta_U,\domU)$ of the non-positive symmetric operator $\Delta:C_c^\infty(U)\rightarrow C_c^\infty(U)$ (densely defined on $L^2(U,\mu)$).

In the next lemma, for any $v\in \domqhat$, we denote by $Ev$ its extension by $0$ in $M\setminus \Omega$.
\begin{lemma}\label{l:ext}
Let $v\in \domqhat$. Then $Ev\in \domqMhat$ and $\widehat{q}_M(Ev)=\widehat{q}_\Omega(v)$.
\end{lemma}
\begin{proof}
Let $v_n\in C_c^\infty(\Omega)$ such that $v_n\rightarrow v$ in $L^2(\Omega,\mu)$ and $q_\Omega(v_n-v_\ell)\rightarrow 0$ as $n,\ell\rightarrow \infty$. We have $Ev_n\rightarrow Ev$ in $L^2(M,\mu)$ and 
$$q_M(Ev_n-Ev_\ell)=(\Delta (Ev_n-Ev_\ell),Ev_n-Ev_\ell)=(\Delta (v_n-v_\ell),v_n-v_\ell)=q_\Omega(v_n-v_\ell) \underset{n,\ell\rightarrow \infty}{\longrightarrow} 0.$$
Thus $Ev\in \domqMhat$ and $\widehat{q}_M(Ev)=\widehat{q}_\Omega(v)$.
\end{proof}

We start the proof of Proposition \ref{p:eigpb} with the following particular case of the subelliptic estimate \cite[Estimate (17.20)]{rs76}: whenever $a$ and $b$ are in $C_c^\infty(M)$ with sufficiently small support and $b=1$ on the support of $a$, there exist $s>0$, $C>0$ such that
$$
\|au\|_{H^s(M)}\leq C\sum_{j=1}^m(\|bX_ju\|_{L^2(M)}+\|bu\|_{L^2(M)}).
$$
Using a partition of unity and the fact that commutators of $X_j$ with smooth cutoff functions are multiplication operators, we can globalize this inequality: there exist $s>0$, $C>0$ such that 
\begin{equation}\label{e:globa}
\|u\|_{H^s(M)}\leq C\sum_{j=1}^m (\|X_ju\|_{L^2(M)}+\|u\|_{L^2(M)}).
\end{equation}
Squaring this inequality, we obtain $\|u\|_{H^s(M)}^2\leq C(-\Delta_M u+u,u)_{L^2(M)}$, which together with the Cauchy-Schwarz inequality implies that 
\begin{equation}\label{e:resolvent}
\|u\|_{H^s(M)}\leq C\|(\text{Id}-\Delta_M)u\|_{L^2(M)}.
\end{equation}

\underline{Case 1: $M$ is compact.} In this case, the Rellich–Kondrachov theorem gives that the resolvent $(\text{Id}-\Delta_M)^{-1}$ is compact from $L^2(M)$ to $L^2(M)$. Using  \cite[Theorem XIII.64 p.245]{reed}, this implies that $\mu_n(\Delta_M)\rightarrow +\infty$ where
\begin{equation}\label{munDeltaM}
\mu_n(\Delta_M)=\inf_{\substack{W\subset \domqMhat\\ \dim(W)=n}} \max_{\substack{v\in W\\ \|v\|=1}} \widehat{q}_M(v).
\end{equation}
Note that in \cite[Theorem XIII.1 p. 76]{reed} the quantity $\mu_n(\Delta_M)$ is defined differently, but it is well-known that $\mu_n(\Delta_M)$ is in fact also equal to \eqref{munDeltaM} (see for instance \cite[Theorem 5.37]{lewin}).
This implies thanks to Lemma \ref{l:ext} that 
$$
\mu_n(\Delta_\Omega)=\inf_{\substack{W'\subset \domqhat\\ \dim(W')=n}} \max_{\substack{v\in W'\\ \|v\|=1}} \widehat{q}_\Omega(v)
$$ 
tends to $+\infty$ as $n\rightarrow +\infty$: indeed, by extending all elements of $W'\subset \domqhat$ by $0$ outside $\Omega$, we obtain an $n$ dimensional subspace $W\subset \domqMhat$, hence $\mu_n(\Delta_M)\leq \mu_n(\Delta_\Omega)$ for any $n\in\N$. Applying again \cite[Theorem XIII.64 p.245]{reed}, this time to $(\Delta_\Omega,\dom)$, we obtain the existence of a complete orthonormal basis $\{\varphi_n\}_{n=1}^\infty$ in $\dom$ so that $-\Delta_\Omega\varphi_n=\lambda_n\varphi_n$ with $0\leq \lambda_1\leq \ldots\leq \lambda_n\leq \ldots\rightarrow +\infty$.

\underline{Case 2: $M=\R^N$.} In this case, recall that we assumed in the introduction that $\Omega$ is bounded and has Lipschitz boundary. The inequality \eqref{e:resolvent} applied to $u\in C_c^\infty(\Omega)$, together with the density of $C_c^\infty(\Omega)$ in $L^2(\Omega)$, proves that $(\text{Id}-\Delta_\Omega)^{-1}:L^2(\Omega)\rightarrow H^s(\Omega)\cap \mathcal{D}(\Delta_\Omega)$ is continuous. Since the injection from $H^s(\Omega)$ to $L^2(\Omega)$ is compact (we use here the fact that the boundary of $\Omega$ is Lipschitz), we obtain directly that the resolvent is compact, and we conclude by \cite[Theorem XIII.64 p.245]{reed}.
 
\section{Proof of Theorem \ref{t:courant}} \label{s:proofcourant}

We start the proof of Theorem \ref{t:courant} with a lemma containing an integration-by-parts formula. We do not give its proof, which follows from the definition of the quadratic form $\widehat{q}_\Omega$ recalled above.
\begin{lemma}\label{l:ipp}
If $u\in \domqhat$, then $X_iu\in L^2(\Omega,\mu)$. Moreover if $u,v\in \domqhat$, then 
$$
\widehat{q}_\Omega(u,v)=-\sum_{i=1}^m \int_\Omega (X_iu)(X_iv)d\mu.
$$
\end{lemma}

We denote by $E_{\lambda_k}$ the eigenspace associated to the eigenvalue $\lambda_k$ of $-\Delta_\Omega$. We have the following min-max principle:
 \begin{lemma}  \label{l:minmax} 
 \begin{enumerate}
 \item $\varphi\in \domqhat\setminus\{0\}$ belongs to $E_{\lambda_1}$ if and only if it minimizes over $\domqhat\setminus\{0\}$ the Rayleigh quotient 
\begin{equation}\label{e:Rayleigh}
 R(\varphi)=\frac{\sum_{i=1}^m \|X_i\varphi\|^2_{L^2(\Omega,\mu)}}{\|\varphi\|_{L^2(\Omega,\mu)}^2}.
\end{equation}
 In this case $R(\varphi)=\lambda_1$.
 \item If $\varphi\in \domqhat\setminus \{0\}$ is orthogonal to $E_{\lambda_1},\ldots,E_{\lambda_{k-1}}$ and $R(\varphi)=\lambda_k$, then $\varphi\in E_{\lambda_k}$. 
 \end{enumerate}
\end{lemma}
The proof is standard and follows for the first point from the computation of $R(\varphi+\varepsilon \psi)$ for $\psi\in \domqhat$ and $\e\rightarrow 0$, and for the second point from the decomposition of $\varphi$ in the orthonormal basis given by Proposition \ref{p:eigpb}.

The next two lemmas are classical in the Riemannian setting but their proofs require some care in the present sub-Riemannian (sR) context.
\begin{lemma}\label{l:krein}
Let $D\subset M$ be a connected open set with $\partial D\neq \emptyset$. Then $\lambda_1(D)>0$, and there exists an eigenfunction of $-\Delta_D$ with eigenvalue $\lambda_1(D)$ which is non-negative.
\end{lemma}
\begin{proof}
Assume for the sake of a contradiction that $\lambda_1(D)=0$ and let $u\not\equiv0$ be an eigenfunction $\Delta u=0$. Then $(\Delta u,u)=0$ hence by definition $\widehat{q}_D(u,u)=0$, which implies $\|X_iu\|_{L^2(D,\mu)}=0$ for any $i$ thanks to Lemma \ref{l:ipp}. But thanks to hypoelliptic regularity \cite{hor} we know that $u\in C^\infty(D)$ (a priori not up to the boundary if $D$ is arbitrary) hence $X_iu\equiv 0$ in $D$. Then $[X_{i_1},[X_{i_2},\ldots]\ldots]u  \equiv 0$ for any bracket of the vector fields, hence by the Hörmander bracket-generating condition $u$ is constant in $D$.

Let us prove that the only constant $u$ which belongs to $\mathcal{D}(\widehat{q}_D)$ is $0$. We choose $(u_\ell)_{\ell\in\N}$ such that $u_\ell \in C_c^\infty(D)$ and $u_\ell\rightarrow u$, $X_iu_\ell\rightarrow 0$ in $L^2(D)$ as $\ell\rightarrow +\infty$. We denote by $\underline{u}$ (resp. $\underline{u}_\ell$) the extension of $u$ (resp. $u_\ell$) to $M$ by $0$ in $M\setminus D$. First, $\underline{u}\in \mathcal{D}(\widehat{q}_M)$ according to Lemma \ref{l:ext}. Let $v\in \mathcal{D}(\widehat{q}_M)$, and $v_\ell\in C_c^\infty(M)$ such that $v_\ell\rightarrow v$ in $L^2(M,\mu)$ and $q_M(v_n-v_\ell)\rightarrow 0$ as $\ell,n\rightarrow +\infty$. Then
$$
\widehat{q}_M(\underline{u},v)=\lim_{\ell\rightarrow +\infty} \sum_{i=1}^m (X_i\underline{u}_\ell,X_iv_\ell)_{L^2(M,\mu)}=\lim_{\ell\rightarrow +\infty}\sum_{i=1}^m (X_iu_\ell,X_iv_\ell)_{L^2(\text{supp}(u_\ell),\mu)}=0
$$
since $(X_iv_\ell)_{\ell\in\N}$ is bounded in $L^2(M,\mu)$, and thus in $L^2(\text{supp}(u_\ell),\mu)$.
Hence $\underline{u}\in \mathcal{D}(\Delta_M)$ and $\Delta_M\underline{u}=0$. By hypoellipticity, $\underline{u}\in C^\infty(M)$, which is impossible since $\underline{u}$ is not smooth across $\partial D\neq \emptyset$.
%
%The subelliptic estimate \cite[(17.20)]{rs76} tells us that there exist $s>0$, $C>0$ such that 
%\begin{equation}\label{e:rs76}
%\|av\|_{H^{s}(\Omega)}\leq C\left(\sum_{i=1}^m\|bX_iv\|_{L^2(\Omega)}+\|bv\|_{L^2(\Omega)}\right)
%\end{equation}
%whenever $a,b\in C_c^\infty(\Omega)$ and $b=1$ on the support of $a$. We deduce that $(u_\ell)_{\ell\in\N}$ is a Cauchy sequence in any $H^\alpha(\Omega)$, hence The right-hand side of \eqref{e:rs76} being equal to $\|b\underline{u}\|_{H^\alpha(\Omega)}$, we get by bootstrap that $\underline{u}\in C^\infty(\Omega)$. If $D\neq\Omega$, this gives $\underline{u}\equiv 0$. If $D=\Omega$, we observe that $\widehat{q}_\Omega(\underline{u},v)=0$ for any $v$ implies that $\underline{u}\in D(\Delta_\Omega)$  and $\Delta_\Omega \underline{u}=0$ a.e. Together with Assumption \ref{a:noncarac}, this implies that $\underline{u}$ is smooth up to the boundary of $\Omega$. Since $\underline{u}$ is constant inside $\Omega$, $\underline{u}\equiv0$. Hence $u\equiv0$, which is a contradiction to $u\not\equiv 0$.

Let $u_0\neq 0$ be in the first eigenspace of $\Delta_D$. Then $|u_0|\in \domqDhat$ and for any $i\in\{1,\ldots,m\}$,
\begin{equation}\label{e:valabsXi}
X_i |u_0|=\begin{cases}
      X_iu_0 &\text{a.e. in } \{u_0>0\} \\
      0 &\text{a.e. in } \{u_0=0\} \\
      -X_iu_0 &\text{a.e. in } \{u_0<0\}
    \end{cases}.
\end{equation}
Both statements follow from \cite[Chapter 5, Exercise 17]{evans}: the main steps are to apply the chain rule to $F_\e\circ u_0$ where $F_\e(z)=\sqrt{z^2+\e^2}-\e$ and then use the dominated convergence theorem.

From \eqref{e:valabsXi} we deduce that $R(|u_0|)=R(u_0)$ where the Rayleigh quotient $R$ is defined in \eqref{e:Rayleigh}. According to Lemma \ref{l:minmax}, this implies that $|u_0|$ is also in the first eigenspace of $\Delta_D$. 
%Hence $u_0^+=(u_0+|u_0|)/2$ and $u_0^-=(|u_0|-u_0)/2$ are also in the first eigenspace. If for both of them  there exists a set on which they are positive, we get a contradiction with the unique continuation of eigenfunctions (implied by Assumption \ref{a:unique}). Hence $u_0^+\equiv0$ or $u_0^-\equiv0$, which implies that $u_0$ does not change sign. Finally, let us prove the simplicity of $\lambda_1(D)$. We assume that $u_0,v_0$ are in the first eigenspace of $\Delta_D$, and that they are not proportional. Without loss of generality, we can assume that they are orthogonal. Thanks to the unique continuation property of eigenfunctions which is satisfied thanks to Assumption \ref{a:unique}, we know that $u_0$ and $v_0$ cannot vanish on any open set. Also, they do not change sign. This contradicts orthogonality and finishes the proof of the lemma.
\end{proof}

\begin{lemma}\label{l:berard}
Let $u\in \dom$ satisfying $-\Delta_\Omega u=\lambda u$ in $\Omega$. Let $D$ be a nodal domain of $u$. Then the restriction $\underline{u}$ of $u$ to $D$ belongs to $\mathcal{D}(\Delta_D)$, and it is an eigenfunction of the Dirichlet problem in $D$, associated to the smallest eigenvalue $\lambda=\lambda_1(D)$.
\end{lemma}
\begin{proof}
Thanks to Assumption \ref{a:noncarac}, $u$ is smooth up to the boundary of $\Omega$ (see \cite[Theorem III, point (4)]{kohnnirenberg}), hence $\underline{u}$ is smooth up to the boundary of $D$. Without loss of generality, we assume that $\underline{u}$ is non-negative. We follow the proof of Lemma 2.0.1 in \cite{anne}, which does not use any regularity on the boundary of the nodal domain. 

We fix a Riemannian metric $g_R$ on $M$, which induces a distance ${\rm dist}_{g_R}$ and a gradient $\nabla_{g_R}$, in order to conveniently conduct our local analysis. Let $\chi_n\in C_c^\infty(\overline{D})$ be a cut-off function such that there exists a constant $C>0$ independent of $n$ such that
\begin{itemize}
\item $\chi_n(x)=1$ for ${\rm dist}_{g_R}(x,\partial D)\geq 1/n$
\item $\chi_n(x)=0$ for ${\rm dist}_{g_R}(x,\partial D)\leq 1/(2n)$
\item $\|\nabla_{g_R} \chi_n\|_{L^\infty}\leq C n$
\item $\|\nabla_{g_R}^2 \chi_n\|_{L^\infty}\leq C n^2$
\end{itemize}
 (the existence of $\chi_n$ is shown in \cite{anne}). Since $X_1,\ldots,X_m$ are smooth, this implies that $\|X_i\chi_n\|_{L^\infty}\leq Cn$ and $\|X_i^*\chi_n\|_{L^\infty}\leq C n$ for any $i$, and $\|\Delta \chi_n\|_{L^\infty}\leq Cn^2$ (the constant $C$ may have changed).

It suffices to show that $\underline{u}$ arises from a Cauchy sequence in $\mathcal{D}(\hat{q}_{\Omega})$ and that $\underline{u}$ generates a continuous linear functional on $L^2(D,\mu)$ via the quadratic form $\hat{q}_{D}$. We set $u_n=\chi_n \underline{u}$. We have $u_n\rightarrow \underline{u}$ in $L^2(D,\mu)$. Let us prove that $q_D(u_n-u_\ell)\rightarrow 0$ as $\ell,n\rightarrow\infty$. To simplify notations, we set $\alpha_{n,\ell}=\chi_n-\chi_\ell$. We have 
\begin{align}
|q_D(u_n-u_\ell)|&=\int_D \sum_{i=1}^m (X_i(u_n-u_\ell))^2 \; d\mu=\int_D \sum_{i=1}^m \alpha_{n,\ell} uX_i^*((X_i\alpha_{n,\ell})u+\alpha_{n,\ell} (X_iu)) d\mu\nonumber\\
&=-\int_D\alpha_{n,\ell} u^2\Delta \alpha_{n,\ell}\; d\mu-\int_D\alpha_{n,\ell}^2u\Delta u\;d\mu\nonumber\\ 
&\qquad \qquad-\int_D\left(\sum_{i=1}^m(X_i\alpha_{n,\ell})(X_i^*u)\alpha_{n,\ell} u +(X_i^*\alpha_{n,\ell})(X_iu)\alpha_{n,\ell} u\right)d\mu\label{e:convqD}\\
&=I_1+I_2+I_3.\nonumber
\end{align}
We show that $I_j\rightarrow 0$ for $j=1,2,3$, a $n,\ell\rightarrow +\infty$. We denote by $A_{n}$ the support of $\nabla\chi_n$, in particular ${\rm vol}(A_n)\rightarrow 0$. We have
\begin{align}
\left|\int_D \alpha_{n,\ell} u^2\Delta \alpha_{n,\ell}\; d\mu\right|\leq C\ell^2\int_{A_{\ell}} u^2 \;d\mu+Cn^2\int_{A_{n}} u^2 \;d\mu.
\end{align}
We use the fact that $u$ is smooth up to the boundary to get that $u^2\leq C\ell^{-2}$ in $A_\ell$ and similarly $u^2\leq Cn^{-2}$ in $A_n$. Hence
$$
|I_1|=\left|\int_{D}\alpha_{n,\ell} u^2\Delta \alpha_{n,\ell} \; d\mu\right|\leq C({\rm vol}(A_\ell)+{\rm vol}(A_n))\underset{n,\ell\rightarrow + \infty}{\longrightarrow} 0.
$$
Then, we have
$$
\left|\int_D\alpha_{n,\ell}^2u\Delta u\;d\mu\right|=\lambda\int_D \alpha_{n,\ell}^2u^2 \;d\mu\leq C({\rm vol}(A_\ell)+{\rm vol}(A_n))\int_D u^2 \;d\mu\underset{n,\ell\rightarrow + \infty}{\longrightarrow} 0.
$$
One can also check that 
$$
I_3=-\frac12\int_D u^2\left(\sum_{i=1}^m\alpha_{n,\ell} X_i^2\alpha_{n,\ell}+(X_i\alpha_{n,\ell})^2+\alpha_{n,\ell}(X_i^*)^2\alpha_{n,\ell}+(X_i^*\alpha_{n,\ell})^2+\alpha_{n,\ell}^2({\rm div}_\mu(X_i))^2\right)d\mu
$$
and once again
$$|I_3|\leq C\ell^2\int_{A_{\ell}} u^2 \;d\mu+Cn^2\int_{A_{n}} u^2 \;d\mu\leq C({\rm vol}(A_\ell)+{\rm vol}(A_n))\underset{n,\ell\rightarrow + \infty}{\longrightarrow} 0.$$
All in all, $|q_D(u_n-u_\ell)|\rightarrow0$ as $n,\ell\rightarrow +\infty$. Hence $\underline{u}\in\mathcal{D}(\widehat{q}_D)$. 

Next, we have to check that $|\widehat{q}_D(\underline{u},v)|\leq C(\underline{u},v)_{L^2(D,\mu)}$ for any $v\in\mathcal{D}(\widehat{q}_D)$. It is sufficient to check it for $v\in C_c^\infty(D)$ and then extend it by density to $\mathcal{D}(\widehat{q}_D)$. Let $v\in C_c^\infty(D)$. We have
\begin{align*}
\widehat{q}_D(\underline{u},v)&=-\lim_{n\rightarrow +\infty} \int_D \sum_{i=1}^m X_i(\chi_nu)X_iv\;d\mu=\lim_{n\rightarrow+\infty}(\Delta u_n,v)_{L^2(D,\mu)}\\
&=(\Delta \underline{u},v)_{L^2(D,\mu)}=-\lambda(\underline{u},v)_{L^2(D,\mu)}
\end{align*}
since $u_n=\underline{u}$ on ${\rm Supp}(v)$ for $n$ sufficiently large. Hence $\underline{u}\in \mathcal{D}(\Delta_D)$ and $\underline{u}$ is an eigenfunction with eigenvalue $\lambda$.

By Lemma \ref{l:krein} we know that $\lambda_1(D)>0$. Assume for the sake of a contradiction that $\lambda>\lambda_1(D)$. Let us denote by $u_0$ a non-negative function in the first eigenspace of $-\Delta_D$, which exists thanks to Lemma \ref{l:krein}. Then according to Proposition \ref{p:eigpb}, $u_0$ and $\underline{u}$ are orthogonal for the $L^2(D,\mu)$ scalar product. At the beginning of the proof, we assumed without loss of generality that $\underline{u}$ is non-negative, but since $D\subset \Omega\setminus Z_{\overline{u}}$, we even know that $\underline{u}$ is strictly positive in $D$. Hence $u_0=0$ a.e. in $D$, which is a contradiction with the fact that $\lambda_1(D)>0$. We conclude that $\lambda_1(D)=\lambda$.
\end{proof}

We start the proof of Theorem \ref{t:courant} by proving its second part. For this, we follow the arguments of \cite[Chapter VI (p. 453-454)]{cou} (see also \cite[Appendix D]{ber}).
Suppose that $u\in E_{\lambda_n}$ has at least $(n + 1)$ nodal domains $D_1,\ldots,D_{n+1}$. We also assume $\lambda_{n-1}<\lambda_n$. For $1\leq i\leq n$, we denote by $u_i$ the restriction of $u$ to $D_i$, which lies in the first eigenspace of the Dirichlet problem in $D_i$ according to Lemma \ref{l:berard}. In particular its Rayleigh quotient $R(u_i)$ is equal to $\lambda_n$ due to Point 1. of Lemma \ref{l:minmax}. We extend $u_i$ by $0$ in $\Omega\setminus D_i$, and we still denote by $u_i$ this extension, which belongs to $\mathcal{D}(\widehat{q}_\Omega)$ according to Lemma \ref{l:ext}.
We can determine $(a_1,\ldots,a_n)\in\R^n\setminus\{0\}$ such that $f=\sum_{i=1}^n a_iu_i$ is orthogonal in $L^2(\Omega,\mu)$ to the $(n-1)$ first eigenfunctions $\varphi_1,\ldots,\varphi_{n-1}$ of $-\Delta_\Omega$ on $\Omega$. We have $R(f)=\lambda_n$, hence $f$ is an eigenfunction for $\lambda_n$ according to the min-max principle (Point 2 of Lemma \ref{l:minmax}). But $f$ vanishes in the open set $D_{n+1}$ in contradiction with the unique continuation property of eigenfunctions which is satisfied thanks to Assumption \ref{a:unique}.

The first part of Theorem \ref{t:courant} follows from similar arguments, except that we avoid using the unique continuation property in the end.\footnote{We would like to thank Iosif Polterovich for sharing with us this argument, itself communicated to him by Dan Mangoubi.} Assume for the sake of a contradiction that $u\in E_{\lambda_n}$ has at least $m_n=n + \text{mult}(\lambda_n)$ nodal domains $D_1,\ldots,D_{m_n}$. By standard linear algebra, there exist $\text{mult}(\lambda_n)+1$ linearly independent functions $f_j$, $j=1,\ldots,\text{mult}(\lambda_n)+1$ of the form  $$f_j=\sum_{i=1}^{m_n} a_{j,i}u_i,$$ with $a_{j,i}\in\R$, $u_i$ is the extension by $0$ of the restriction of $u$ to $D_i$, and $f_j$ is orthogonal in $L^2(\Omega,\mu)$ to the $(n-1)$ first eigenfunctions $\varphi_1,\ldots,\varphi_{n-1}$ of $-\Delta_\Omega$ on $\Omega$. For any $j$ we have $R(f_j)=\lambda_n$, hence $f_j$ is an eigenfunction for $\lambda_n$ according to the min-max principle (Point 2 of Lemma \ref{l:minmax}). The $f_j$ are $\text{mult}(\lambda_n)+1$ linearly independent eigenfunctions with eigenvalue $\lambda_n$, which is impossible.

\begin{remark}
Counterexamples to the unique continuation property are known when Assumption \ref{a:unique} is not satisfied, for operators of the form $\Delta+V$ where $V$ is a smooth function on $\Omega$ (see \cite{bah}).
\end{remark}

\section{Sub-Riemannian tools} \label{s:sRtools}
This section introduces the notations, the terminology and the tools of sub-Riemannian (sR) geometry which will be needed in the proof of Theorem \ref{t:density}. For a more comprehensive introduction to sR geometry, we refer to \cite{agrachev2019} and \cite{bellaiche1996tangent}.

\subsection{Sub-Riemannian flag}  \label{s:sRflag}
First, we define the sR \textit{distribution}
$$
\mathcal{D}={\rm Span}(X_1,\ldots,X_m)
$$
and then the sR \textit{flag} as follows: $\mathcal{D}^0=\{ 0\}$, $\mathcal{D}^1=\mathcal{D}$, and, for any $j\geq 1$, $$\mathcal{D}^{j+1}=\mathcal{D}^j+[\mathcal{D},\mathcal{D}^j].$$ For any $q\in M$, this gives a flag
\begin{equation}\label{e:sRflag37}
\{0\}=\mathcal{D}_q^0\subset \mathcal{D}_q^1\subset \ldots \subset \mathcal{D}_q^{r-1} \varsubsetneq \mathcal{D}_q^{r(q)} =T_qM.
\end{equation}
where $\mathcal{D}^i_q$ denotes $\mathcal{D}^i$ taken at point $q$. The integer $r(q)$ is called the \textit{step}, or \textit{non-holonomic order}, of $\mathcal{D}$ at $q$. In the case of Example \ref{exGrushin} it is equal to $1$ except on the singular line $\{x=0\}$, where it is equal to $2$ (or $\alpha$). In the case of Example \ref{exCarnot} it coincides with the step $r$ defined there; in particular it is equal to $2$ in Example \ref{exHeis} at any point.

For $i\in\{0,\ldots,r(q)\}$, we set $$n_i(q)=\dim \mathcal{D}_q^i.$$ The sequence $(n_i(q))_{0\leq i\leq r(q)}$ is called the \textit{growth vector} at point $q$. We say that $q\in M$ is \textit{regular} if the growth vector $(n_i(q'))_{0\leq i\leq r(q')}$ at $q'$ is constant for $q'$ in a neighborhood of $q$. Otherwise, $q$ is said to be \textit{singular}. If any point $q\in M$ is regular, we say that the structure is \textit{equiregular}.

\begin{remark}
The Heisenberg sub-Laplacian of Example \ref{exHeis} is equiregular (with $n_1=2d$, $n_2=2d+1$), but the Baouendi-Grushin sub-Laplacian of Example \ref{exGrushin} is not equiregular.
\end{remark}

The number
\begin{equation} \label{e:homogdim}
\mathcal{Q}(q)=\sum_{i=1}^{r(q)}i(n_i(q)-n_{i-1}(q)),
\end{equation} 
coincides at any regular point $q$ with the Hausdorff dimension of the metric space induced by the sR distance on $M$ near $q$ (see \cite{mitchell85}).

We define a non-decreasing sequence of weights $w_i(q)$. Roughly speaking, $w_i(q)$ is the minimal length of the brackets of $X_1,\ldots,X_m$ needed to generate $i$ independent directions at $q$. Formally, given any $i\in \{1,\ldots,N\}$ and $q\in M$, there exists a unique $j\in\{1,\ldots,r(q)\}$ such that $n_{j-1}(q)+1\leq i\leq n_j(q)$. We set $w_i(q)=j$. It is not difficult to check that \eqref{e:homogdim} coincides with $\sum_{i=1}^{N}w_i(q)$ (see \eqref{e:premierQ}).

\subsection{Nilpotentization} \label{s:nilpotproced}

The aim of the following paragraphs is to introduce a system of local coordinates, called privileged coordinates, in which it is natural to write Taylor expansions of vector fields defined on the sR manifold (see \cite[Section 4]{bellaiche1996tangent}, \cite[Chapter 2]{jean2014control}). The first order term in the Taylor expansion of a vector field in privileged coordinates is called the nilpotent approximation of the vector field. A typical example of privileged coordinates system is given by some exponential coordinates with respect to a frame of $T_qM$ which is ``adapted'' to the sR flag \eqref{e:sRflag37} (see \eqref{e:privicoordexpo2} below).

\subsubsection{Non-holonomic orders.} 
The \textit{non-holonomic order} of a smooth germ of function is

\begin{equation*}
\ord_q(f)=\min\{ p\in\mathbb{N} : \exists i_1,\ldots,i_p\in \{1,\ldots,m\} \text{ such that } (X_{i_1}\ldots X_{i_p}f)(q)\neq 0\}
\end{equation*}
where we adopt the convention that $\min\emptyset = +\infty$.

 The \textit{non-holonomic order} of  a smooth germ of vector field $X$ at $q$, denoted by $\ord_q(X)$, is the real number
 
\begin{equation*}
\ord_q(X)=\sup\{\sigma\in\R : \ord_q(Xf)\geq \sigma+\ord_q(f), \ \ \forall f \in C^\infty(q)\}\in\Z.
\end{equation*} 
In other words, applying $X$ to a function $f$ ``increases'' the non-holonomic order by at least $\ord_q(X)$ (we put quotation marks to indicate that since $\ord_q(X)$ may be negative, applying $X$ may in fact decrease the non-holonomic order).

There holds $\ord_q(fg)=\ord_q(f)+\ord_q(g)$, $\ord_{q}(fX)= \ord_q(f)+\ord_q(X)$ and $\ord_q([X,Y])\geq \ord_q(X)+\ord_q(Y)$. As a consequence, every $X$ which has the property that $X(q')\in\mathcal{D}^i_{q'}$ for any $q'$ in a neighborhood of $q$ is of non-holonomic order $\geq -i$.

\begin{example}
Let us illustrate these definitions on Example \ref{exGrushin} where $X_1=\partial_x$ and $X_2=x\partial_y$. At $(x,y)=0$, the non-holonomic order of the function $f(x,y)=x$ is $1$ and the non-holonomic order of $f(x,y)=y$ is $2$ since $X_1X_2y=1$. The non-holonomic order of $\partial_x$ is $-1$, the non-holonomic order of $\partial_y$ is $-2$, and the non-holonomic order of $x\partial_y$ is $-1$.
\end{example}

\subsubsection{Privileged coordinates.}

A system of \textit{privileged coordinates} at $q$ is a system of local coordinates $(x_1,\ldots,x_N)$ near $q$ verifying
\begin{equation}\label{e:nonholpriv}
\ord_q(x_i)=w_i, \qquad \text{ for $1\leq i\leq N$}.
\end{equation}
 In particular, privileged coordinates satisfy $\partial_{x_i}\in \mathcal{D}_q^{w_i(q)}\backslash \mathcal{D}_q^{w_i(q)-1}$ at $q$, meaning that privileged coordinates are adapted to the flag (see definition below). One can also check that for any $(\alpha_1,\ldots,\alpha_N)\in\N^N, (\beta_1,\ldots,\beta_N)\in\N^N$, 
\begin{equation}\label{e:ordermonomials}
 \ord_q(x_1^{\alpha_1}\ldots x_N^{\alpha_N}\partial_{x_1}^{\beta_1}\ldots\partial_{x_N}^{\beta_N})=\sum_{i=1}^N (\alpha_i-\beta_i)w_i.
\end{equation}

We now describe a construction showing that privileged coordinates systems exist at any $q\in M$. A family $(Z_1,\ldots,Z_N)$ of vector fields is said to be \textit{adapted} to the sR flag at $q$ if it is a frame of $T_qM$ at $q$ and if $Z_i(q)\in \mathcal{D}_q^{w_i(q)}$ for any $i\in \{1,\ldots,N\}$. In other words, for any $i\in\{1,\ldots,r(q)\}$, the vectors $Z_1, \ldots, Z_{n_i(q)}$ at $q$ span $\mathcal{D}_q^i$.

If  $(Z_1,\ldots,Z_N)$ is an adapted frame at $q$, it is proved in \cite[Appendix B]{jean2014control} that the inverse of the local diffeomorphism
\begin{equation}\label{e:privicoordexpo2}
(x_1,\ldots,x_n)\mapsto \exp(x_1Z_1)\ldots \exp(x_N Z_N)(q)
\end{equation}
defines privileged coordinates at $q$ (called exponential coordinates of the second kind).

\subsubsection{Dilations.} 
As we mentioned in the introduction, sR geometries are anisotropic. The natural sR dilations that we now define are thus also anisotropic. 

Fix $q\in M$. For every $\varepsilon \in \mathbb{R}\setminus \{0\}$, the dilation $\delta_\e:\R^N\rightarrow \R^N$ is defined by

\begin{equation*}
\delta_\varepsilon(x)=(\varepsilon^{w_1(q)}x_1,\ldots,\varepsilon^{w_N(q)}x_N)
\end{equation*}
for every $x=(x_1,\ldots,x_N)$ - we omit the dependance in $q$ in the notation.

A dilation $\delta_\e$ acts also on functions and vector fields on $\R^N$ by pull-back: $\delta_\e^*f=f\circ \delta_\e$ and $\delta_\e^*X$ is the vector field such that $(\delta_\e^*X)(\delta_\e^*f)=\delta_\e^*(Xf)$ for any $f\in C^1(\R^N)$.

In particular, given a system of privileged coordinates $\psi_q:U\rightarrow\R^N$, for any vector field $X$ in $U$ of non-holonomic order $k$ there holds $\delta_\e^*(\psi_q)_*X=\e^{-k}(\psi_q)_*X$. We will use this property many times for vector fields of the form \eqref{e:ordermonomials}.

\subsubsection{Nilpotent approximation.}\label{s:nilllll} We now turn to the definition of the nilpotent approximation, which is a first-order approximation of vector fields in privileged coordinates near a point $q\in M$. An explicit example of computation of nilpotent approximation is given in \cite[Example 2.8]{jean2014control}, it may help to understand the definitions which follow.

Fix a system of privileged coordinates $\psi_q=(x_1,\ldots,x_N):U\rightarrow\R^N$ defined in a neighborhood $U$ of $q$. Coming back to the vector fields $X_1,\ldots,X_m$, we write the Taylor expansion

\begin{equation}\label{e:Taylorexpansionvf}
(\psi_q)_*X_i(x)\sim \sum_{\substack{\alpha\in \N^N\\ j\in \{1,\ldots,N\}}}a_{\alpha,j}x^\alpha \partial_{x_j}.
\end{equation}

Since $X_i\in\mathcal{D}$, its non-holonomic order is $-1$. Hence, when  $a_{\alpha,j}\neq 0$ for some $\alpha=(\alpha_1,\ldots,\alpha_N)\in \N^N$ and $j\in\{1,\ldots,N\}$,  
the monomial vector field $x^\alpha\partial_{x_j}$ has non-holonomic order $\geq -1$, which implies that $\sum_{i=1}^N w_i(q)\alpha_i\geq w_j(q)-1$ according to \eqref{e:ordermonomials}. Therefore, we may write $X_i$ as a formal series

\begin{equation}\label{e:formalseries}
(\psi_q)_*X_i=X_i^{(-1)}+X_i^{(0)}+X_i^{(1)}+\ldots
\end{equation}
where $X_i^{(k)}$ is a homogeneous vector field of degree $k$, meaning that

\begin{equation}\label{e:homodunilpo}
\delta_\varepsilon^*X^{(k)}_i=\varepsilon^{k}X^{(k)}_i.
\end{equation}
We set 
\begin{equation}\label{e:nilpoasfirst}
\widehat{X}^q_i=X_i^{(-1)}, \qquad 1\leq i\leq m
\end{equation}
which is a vector field on $\R^N$. Then $\widehat{X}^q_i$ is homogeneous of degree $-1$ with respect to dilations, meaning that $\delta_\varepsilon^* \widehat{X}^q_i=\varepsilon^{-1}\widehat{X}^q_i$ for $\varepsilon\neq 0$. For $\e>0$ small enough we have
\begin{equation} \label{e:withremainder}
X_i^\e:= \e\delta_{\e}^*(\psi_q)_*X_i=\widehat{X}_i^q+\e R_{i,\e}^q
\end{equation}
where $R_{i,\e}^q$ depends smoothly on $\e$ for the $C^\infty$ topology (see also \cite[Lemma 10.58]{agrachev2019}). 

Finally, the nilpotent approximation of $X_1,\ldots,X_m$ at $q$ is defined as $\widehat{M}^q\simeq\R^N$ endowed with the vector fields $\widehat{X}_1^q,\ldots,\widehat{X}^q_m$. This definition does not depend on the choice of privileged coordinates at $q$ because two sets of such coordinates produce two ``sR-isometric'' sR structures.   An important property is that $(\widehat{X}^q_1,\ldots,\widehat{X}^q_m)$ generates a nilpotent Lie algebra of step $r(q)$ (see \cite[Proposition 2.3]{jean2014control}).

 The nilpotent approximation of a measure $\mu$ on $M$ at $q\in M$ is the measure on $\R^N$
 \begin{equation}\label{e:nilpomeasure}
 \widehat{\mu}^q=\lim_{\e\rightarrow 0} \e^{-\mathcal{Q}(q)}\delta_\e^*(\psi_q)_*\mu
 \end{equation}
 where the convergence is understood in the vague topology. It follows from this definition that $\widehat{\mu}^q$ is proportional to the Lebesgue measure.
 
 \subsection{Desingularization}
The estimates we will need at some point in the proof of Theorem \ref{t:density} blow-up at singular points. However, when $q\in M$ is a singular point, it is possible to lift locally in a neighborhood $U$ of $q$ the vector fields $X_1,\ldots,X_m$ to vector fields $\widetilde{X}_1,\ldots,\widetilde{X}_m$ on  $\widetilde{U}=U\times \R^K$, so that the lift $\widetilde{q}=(q,0)$ of $q$ is a \textit{regular point} in $\widetilde{U}$ and many properties of the vector fields are preserved. This lifting procedure will allow us to recover uniform estimates in Section \ref{s:endofth2}.

\begin{lemma}\cite[Lemma 2.5 and Theorem 2.9]{jean2014control} \label{l:desingularisation}
Let $q$ be a point in $M$. Then there exist $K\in\N$, a neighborhood $U\subset M$ of $q$, coordinates $(x,y)$ on $\widetilde{U}=U\times \R^{K}$ and smooth vector fields 
$$
\widetilde{X}_i(x,y)=X_i(x)+\sum_{j=1}^K b_{ij}(x,y)\partial_{y_j}, \qquad i=1,\ldots,m,
$$
on $\widetilde{U}$ such that
\begin{itemize}
\item $\widetilde{X}_1,\ldots,\widetilde{X}_m$ satisfy Hörmander's bracket-generating condition in $\widetilde{U}$;
\item every $\widetilde{p}$ in $\widetilde{U}$ is regular;
\item denoting by $\pi:\widetilde{U}\rightarrow U$ the canonical projection, and by $\widetilde{d}$ the sR distance defined by $\widetilde{X}_1,\ldots,\widetilde{X}_m$ on $\widetilde{U}$, we have $\pi_*\widetilde{X}_i=X_i$, and for $p\in U$ and $\e$ small enough,
\begin{equation}\label{e:projball}
B_\e(p)=\pi\left(B_\e^{\widetilde{d}}((p,0))\right).
\end{equation}
\end{itemize}
\end{lemma}

\begin{example}
A possible (global) desingularization of the vector fields $X_1=\partial_x$ and $X_2=x\partial_y$ on $\R^2$ is given by the vector fields $\widetilde{X}_1=\partial_x$ and $\widetilde{X}_2=\partial_z+x\partial_y$ on $\R^3$.
\end{example}

\section{Proof of Theorem \ref{t:density}} \label{s:theorem2}

The proof of Theorem \ref{t:density} splits into two steps. The first one consists in proving an asymptotic upper bound for the first eigenvalue of the Dirichlet sub-Laplacian in a sR ball centered at $q\in M$ whose radius tends to $0$. This upper bound is uniform in $q$ when $q$ is a regular point. The second step is to use Lemma \ref{l:desingularisation} (i.e., a desingularization) to conclude.

\subsection{The first eigenvalue of the sub-Laplacian in a small sR ball}
We fix $q\in M$ and we take a chart $\psi_q:U\rightarrow\R^N$ of privileged coordinates at $q$, with $\psi_q(q)=0$. 
%Pushing forward to $\R^N$ the vector fields $X_i$ (resp. the measure $\mu$), we can consider them as vector fields (resp. a measure) on $U$. In the sequel, we also push-forward the distance, the balls etc. In particular the push-forward of the ball $B_r(q)$ defined in \eqref{e:ball} is denoted by $B_{r,q}\subset \R^N$, and it is centered at $0\in\R^N$. 
We denote by $R(q)$ the maximal radius such that $B_{R(q)}(q)\subset U$.

As seen in Section \ref{s:nilllll}, the nilpotent approximations of $X_1,\ldots,X_m$ (resp. of $\mu$) at $q$ are vector fields $\widehat{X}_i^q$ (resp. a measure $\widehat{\mu}^q$) in $\R^N$. 

%We denote by $\Delta_{r,q}$ the Dirichlet sub-Laplacian in $B_{r,q}$, constructed with the vector fields $X_1,\ldots,X_m$ and the measure $\mu$. 
For $\e\leq R(q)$, we set $$L^2_{\e}=L^2((\psi_q)_*B_\e(q),(\psi_q)_*\mu).$$
We also fix $u_1\in C_c^\infty(\Box(1/2))$ such that  $u_1(0)\neq 0$. Finally, we set 
\begin{equation}\label{e:ur(x)}
u_\e(x)=\e^{-\mathcal{Q}(q)/2}u_1(\delta_{1/\e}x)
\end{equation}
and we have $u_\e\in C_c^\infty(\Box(\e/2))$.
\begin{lemma} \label{l:eignilpo}
If $X$ is a vector field on $\psi_q(U)\subset\R^N$ which is homogeneous of degree $k\in \R$ (in the sense of \eqref{e:homodunilpo}), then there exist $c(q)>0, \varepsilon(q)>0$ such that for any $\e\leq \varepsilon(q)$ and any $1\leq i\leq m$,
\begin{equation}\label{e:Xiur}
\|Xu_\e\|_{L^2_{\e}}\leq c(q)\e^k\|u_\e\|_{L^2_{\e}}
\end{equation}
\end{lemma}
%\begin{remark}
%Note that the norms involved in \eqref{e:Xiur} are the $L^2_{r}$ norms and not the $L^2(\widehat{B}_{r,q},\widehat{\mu}^q)$ norms, which would have been more ``natural'' to state an inequality like \eqref{e:Xiur}.
%\end{remark}
\begin{proof}

First, the ball-box theorem \eqref{e:bboxthe} (see also \cite{bellaiche1996tangent}, \cite[Corollary 2.1]{jean2014control}) yields the existence of $0<\alpha\leq 1$ and $\varepsilon(q)>0$ (both depending on $q$) such that for any $0<\varepsilon<\varepsilon(q)$,
\begin{equation}\label{e:bbox}
\Box(\alpha\e)\subset (\psi_q)_*B_{\e}(q)\subset \Box(\alpha^{-1}\e).
\end{equation}

We use \eqref{e:ur(x)} and the homogeneity in $\e$ of $X$, $\Box(\e)$ and $\widehat{\mu}^q$.
Due to \eqref{e:nilpomeasure} this implies the following two convergences
\begin{align*}
\e^{-\mathcal{Q}(q)/2}\e^{-k}\|Xu_\e\|_{L^2(\Box(\e),(\psi_q)_*\mu)}&=\e^{-\mathcal{Q}(q)/2}\|Xu_1\|_{L^2(\Box(1),\delta_\e^*(\psi_q)_*\mu)}\underset{\e\rightarrow 0}{\longrightarrow}  \|Xu_1\|_{L^2(\Box(1),\widehat{\mu}^q)}\\
\e^{-\mathcal{Q}(q)/2}\|u_\e\|_{L^2(\Box(\alpha^{2}\e),(\psi_q)_*\mu)}&=\e^{-\mathcal{Q}(q)/2}\|u_1\|_{L^2(\Box(\alpha^{2}),\delta_\e^*(\psi_q)_*\mu)}\underset{\e\rightarrow 0}{\longrightarrow}  \|u_1\|_{L^2(\Box(\alpha^{2}),\widehat{\mu}^q)}.
\end{align*}
Taking the ratio of the two convergences (justified by the fact that $u_1(0)\neq 0$ hence the last limit is $\neq 0$), we obtain
\begin{equation}\label{e:afterchangemeasure}
\|Xu_\e\|_{L^2(\Box(\e),(\psi_q)_*\mu)}\leq c(q)\e^k\|u_\e\|_{L^2(\Box(\alpha^2\e),(\psi_q)_*\mu)}.
\end{equation}
Using twice \eqref{e:bbox}, we obtain
$$
\|Xu_\e\|_{L^2_{\e}}\leq  \|Xu_\e\|_{L^2(\Box(\alpha^{-1} \e),(\psi_q)_*\mu)}\leq c(q)\e^k\|u_\e\|_{L^2(\Box(\alpha \e),(\psi_q)_*\mu)}\leq c(q)\e^k\|u_\e\|_{L^2_{\e}}
$$
which implies the lemma.
\end{proof}

\begin{corollary} \label{c:upperbound}
For any $q\in M$, there exist $c(q)>0$ and $\e(q)>0$ such that for any $\e\leq\e(q)$, there holds
$\lambda_1(B_\e(q))\leq c(q)\e^{-2}$. 
\end{corollary}
\begin{proof}
We fix $q\in M$ and $1\leq i\leq m$.  According to \eqref{e:bbox}, we know that $u_{\alpha \e}$ is supported in $(\psi_q)_*B_\e(q)$.  
We set $R_i^q=(\psi_q)_*X_i-\widehat{X}_i^q$, which is a vector field on $(\psi_q)_*(U)\subset \R^N$. Let us assume for the moment that we have proved the existence of $C(q)$ such that for $\e$ sufficiently small,
\begin{equation}\label{e:lowerorderterm}
\|R_i^qu_{\alpha \e}\|_{L^2_\e}^2\leq C(q)\|u_{\alpha \e}\|_{L^2_\e}^2.
\end{equation} 
Then we can write
\begin{align}
\|(\psi_q)_*X_iu_{\alpha \e}\|_{L^2_{\e}}^2&\leq 2\left(\|\widehat{X}_i^qu_{\alpha \e}\|_{L^2_\e}^2+\|R_i^qu_{\alpha \e}\|_{L^2_\e}^2\right)\nonumber\\
&\leq (c(q)\e^{-2}+C(q))\|u_{\alpha \e}\|_{L^2_\e}^2\label{e:ami}\\
&\leq c'(q)\e^{-2}\|u_{\alpha \e}\|_{L^2_\e}^2\nonumber
\end{align}
where in the second inequality we used Lemma \ref{l:eignilpo} with $X=\widehat{X}_i^q$, and \eqref{e:lowerorderterm}. By the min-max principle (Lemma \ref{l:minmax}), we get the result.

There remains to prove \eqref{e:lowerorderterm}. We can write $R_{i}^q=\sum_{k=1}^N a_{i,k}^q \partial_{x_k}$ where the $a_{i,k}^q$ are smooth functions of $x\in \R^N$.  We then apply Taylor's theorem for multivariate functions with exact remainder to each $a_{i,k}^q$. Recalling \eqref{e:formalseries}, \eqref{e:nilpoasfirst} and the fact that $R_i^q=(\psi_q)_*X_i-\widehat{X}_i^q$ has homogeneous components of order $\geq 0$ only, this yields a decomposition (``factorizing out from $R_i^q$ monomial vector fields of order $0$'')
\begin{equation}\label{e:bijq}
R_{i}^q=\sum_{j\in\J}b_{i,j}^qY_j
\end{equation}
where $b_{i,j}^q$ is a continuous function of $x$. Here $\mathcal{J}$ is a finite set such that $(Y_j)_{j\in\J}$ consists of all monomial vector fields $Y$ which are homogeneous of degree $0$, i.e.  $Y$ is of the form $Y=x_i\partial_{x_k}$ for some indices $i,k\in\{1,\ldots,n\}$ satisfying $w_i=w_k$. 
Then we can write
$$
\|R_i^qu_{\alpha \e}\|_{L^2_\e}\leq  \sum_{j\in\J} \|b_{i,j}^q Y_ju_{\alpha \e}\|_{L^2_\e}\leq  \sum_{j\in\J} \|b_{i,j}^q\|_{L^\infty((\psi_q)_*B_\e(q))} \| Y_ju_{\alpha \e}\|_{L^2_\e}\leq C(q)\|u_{\alpha \e}\|_{L^2_\e}
$$
where the last inequality comes from Lemma \ref{l:eignilpo} applied with $k=0$ (and the fact that $b_{i,j}^q$ is smooth, hence bounded). This concludes the proof of \eqref{e:lowerorderterm}.
\end{proof}

\begin{remark}
One can in fact prove that $\e^2\lambda_1(B_\e(q))$ converges to the first eigenvalue of the Dirichlet sub-Laplacian $\widehat{\Delta}^q$ on $L^2(\widehat{B}_1(q),\widehat{\mu}^q)$, which is a stronger statement than Corollary \ref{c:upperbound}. Here, $\widehat{\Delta}^q=\sum_{i=1}^m (\widehat{X}_i^q)^2$ and $\widehat{B}_1(q)\subset\R^N$ denotes the sR ball computed with the metric obtained by replacing in \eqref{e:me} the vector fields $X_i$ by the nilpotentized ones $\widehat{X}_i^q$.

However, what we will need for our purpose is the uniformity of the convergence with respect to $q$, and it is easier to prove the uniformity of $c(q)$ with respect to $q$ (see Lemma \ref{l:uniformconstant}) than the uniformity of the convergence of $\e^2\lambda_1(B_\e(q))$ with respect to $q$. This is why we prefered to keep our weaker statements.
\end{remark}

\begin{lemma} \label{l:uniformconstant}
When $q$ is regular, the constants $c(q)$ and $\e(q)$ in Corollary \ref{c:upperbound} can be taken uniform in a small neighborhood of $q$.
\end{lemma}
\begin{proof}
We use the fact that taking a nilpotent approximation is a ``uniform'' procedure near a regular point (but it is not uniform near a singular point). This fact is described in Section 2.2.2 in \cite{jean2014control}, and it mainly relies on the property that $q$ being regular, there exists a smooth frame $q'\mapsto (Z_1(q'),\ldots,Z_N(q'))\in (T_{q'}M)^N$ which is an adapted frame at every point $q'$ in some neighborhood $V$ of $q$. Using \eqref{e:privicoordexpo2}, this yields a smoothly varying system of privileged coordinates in $V$, and a smooth nilpotent approximation in $V$ (see Definition 2.9 in \cite{jean2014control}). As remarked in \cite[Section 4.4]{CdVHT} where a similar uniformity argument as ours is carried out, the continuity in Definition 2.9 of \cite{jean2014control} can be replaced by smoothness.

It follows from Theorem 2.3 in \cite{jean2014control} that the constants $\alpha$ and $\varepsilon(q')$ in \eqref{e:bbox} can be taken uniform over $q'\in V$. These uniform constants are respectively denoted by $\alpha\in(0,1]$ and $\varepsilon(V)>0$. We will deduce that the inequality  \eqref{e:afterchangemeasure} remains true in $V$ with a uniform constant $c(V)$. To state this property rigorously (see \eqref{e:unifff}), we notice that given any family of vector fields $V\ni q'\mapsto X^{q'}$ which is smooth in $q'\in V$, since $\{\widehat{\mu}^{q'}\}_{q'\in V}$ is a smooth family of measures (see \cite[Section 4.1]{abb12}), the constant 
$$c(q')=\frac{\|X^{q'}u_1\|_{L^2(\Box(1),\widehat{\mu}^{q'})}}{\|u_1\|_{L^2(\Box(\alpha^2),\widehat{\mu}^{q'})}}$$ in \eqref{e:afterchangemeasure}  is continuous over $V$, and its supremum over $V$ is denoted by $c(V)$. Moreover, the convergence \eqref{e:nilpomeasure} is also uniform over $V$, due to the smoothness of $\mu$. We assume that each vector field $X^{q'}$ is homogeneous of degree $k\in\R$. Following the proof of Lemma \ref{l:eignilpo}, we obtain that for any $q'\in V$ and any $0<\e<\e(V)$ there holds
\begin{equation}\label{e:unifff}
\|X^{q'}u_\e\|_{L^2(\Box(\e),(\psi_{q'})_*\mu)}\leq c(V)\e^{k}\|u_\e\|_{L^2(\Box(\alpha^{2} \e),(\psi_{q'})_*\mu)}.
\end{equation}
Applying \eqref{e:unifff} to $X^{q'}=\widehat{X}_i^{q'}$ we obtain that $\|\widehat{X}_i^{q'}u_{\alpha \e}\|_{L^2_\e}\leq c(V)\varepsilon^{-2}$ for any $q'\in V$.  

Since $q'\mapsto (\psi_{q'})_*X_i(q')$ and $q'\mapsto \widehat{X}_i^{q'}$ are smooth, the map $$V\ni q'\mapsto (\psi_{q'})_*X_i(q')-\widehat{X}_i^{q'}(q')=R_i^{q'}$$ is also smooth. Therefore the functions $b_{i,j}^{q'}$ defined in \eqref{e:bijq} depend smoothly on $q'$. Besides, the vector fields $x_i\partial_{x_k}$ for $i,k\in\{1,\ldots,N\}$, which appear in \eqref{e:bijq}, do not depend on $q'$. This implies that the constant $C(q')$ in \eqref{e:lowerorderterm} can be taken uniform over $V$. More precisely, this means that there exists $C(V)>0$ such that for any $\e<\e(V)$ and any $q'\in V$,
\begin{equation}\label{e:lowerorderterm2}
\|R_i^{q'}u_{\alpha \e}\|_{L^2((\psi_{q'})_*B_\e(q'),(\psi_{q'})_*\mu)}^2\leq C(V)\|u_{\alpha \e}\|_{L^2((\psi_{q'})_*B_\e(q'),(\psi_{q'})_*\mu)}^2.
\end{equation} 
This implies that the constants $c(q'), \e(q')$ in Corollary \ref{c:upperbound} can be taken uniform over $q'\in V$, which proves Lemma \ref{l:uniformconstant}.
\end{proof}

\begin{remark}
Corollary \ref{c:upperbound} establishes an upper bound for the first \textit{Dirichlet} eigenvalue as $r\rightarrow 0$. Note that lower bounds on the first \textit{Neumann} eigenvalue were established in \cite{jer2}, this is equivalent to Poincaré's inequality.
\end{remark}

\subsection{End of the proof of Theorem \ref{t:density}} \label{s:endofth2}

We assume that $\varphi_\lambda$ is an eigenfunction of $-\Delta_\Omega$ not belonging to the first eigenspace $E_{\lambda_1}$. We denote by $D_j$ its nodal domains. According to Lemma \ref{l:berard}, the restriction of $\varphi_\lambda$ to each domain $D_j$ is an eigenfunction of the Dirichlet sub-Laplacian $\Delta_{D_j}$, it belongs to its first eigenspace, and $\lambda_1(D_j)=\lambda$ for each $D_j$. Now, if $x\in \Omega$ and $d(q,Z_{\varphi_\lambda})>\e$, where $d$ is the sR distance, then $B_\e(q)\subset D_j$ for some $j$. By the min-max principle, it implies that $\lambda=\lambda_1(D_j)\leq\lambda_1(B_\e(q))$. But $\lambda_1(B_\e(q))\leq c(q)\e^{-2}$ thanks to Corollary \ref{c:upperbound}, hence $\e\leq c(q)\lambda^{-1/2}$.

If $q$ is regular, using Lemma \ref{l:uniformconstant}, we obtain that the constant $c(q')$ above is in fact uniform for $q'$ in a neighborhood of $q$. Hence any sR ball centered in a neighborhood $V$ of $q$ and of radius $\geq c'(V)\lambda^{-1/2}$ will intersect $Z_{\varphi_\lambda}$, which concludes the proof of the theorem ``locally near $q$'' in this case.

If $q\in \Omega$ is a singular point, the idea is to desingularize the vector fields at $q$ thanks to Lemma \ref{l:desingularisation} in order to recover a regular neighborhood but in a higher-dimensional space, and be able to apply the result we just obtained in the regular case. Following the notations and defined quantities of Lemma \ref{l:desingularisation}, we consider
$$
\widetilde{X}_i(x,y)=X_i(x)+\sum_{j=1}^K b_{ij}(x,y)\partial_{y_j}, \qquad i=1,\ldots,m,
$$
for $(x,y)\in U\times \R^K$.

We will build a sub-Laplacian $\widetilde{\Delta}$ satisfying the following key properties:
\begin{itemize}
\item It is defined on the bounded set $\Omega\times \mathbb{T}^K$ where $\mathbb{T}=\R/2\pi\Z$.
\item In a neighborhood of $(q,0)$ we have $\widetilde{\Delta}=-\sum_{i=1}^m\widetilde{X}_i^*\widetilde{X}_i$. Its domain $D(\widetilde{\Delta})$ is constructed as in Section \ref{s:proofprop}.
\item The vector fields defining $\widetilde{\Delta}$ satisfy Hörmander's bracket-generating condition everywhere in $\Omega\times \mathbb{T}^K$.
\end{itemize}
The construction of $\widetilde{\Delta}$ is achieved through cut-offs and extensions of the vector fields $\widetilde{X}_i$ to vector fields which are periodic in the $y_j$ variables and thus defined on $\Omega\times \mathbb{T}^K$. In the sequel, $\mathbb{T}$ is identified with $[-\pi, \pi)$ (with periodic boundary).

Without loss of generality we assume that $U$ is contained in the fundamental domain $\Omega\times [-\pi,\pi)^K$. We fix a compact set $V\subset U$ which is a neighborhood of $(q,0)$. We consider cut-off functions satisfying:
\begin{itemize}
\item $\chi_0:\Omega\times\T^K\rightarrow \R^+$ is a smooth function which is equal to $1$ in $V$ and $0$ in $\Omega\times\T^K\setminus U$. 
\item $\chi_1:\Omega\times\T^K\rightarrow \R^+$ is a smooth function which is equal to $0$ in $V$ and $>0$ outside $V$.
\end{itemize}
We consider the vector fields 
\begin{equation}\label{e:theunderlined}
\underline{\widetilde{X}}_i(x,y)=X_i+\sum_{j=1}^K\chi_0(x,y)b_{ij}(x,y)\partial_{y_j}
\end{equation}
for $i\in\{1,\ldots,m\}$ and  $Y_j=\chi_1(x,y)\partial_{y_j}$ for $j\in\{1,\ldots,K\}$ on $\Omega\times \mathbb{T}^K$, and the sub-Laplacian on $\Omega\times \mathbb{T}^K$ defined by 
$$
\widetilde{\Delta}=-\sum_{i=1}^m (\underline{\widetilde{X}}_i)^*\underline{\widetilde{X}}_i-\sum_{j=1}^KY_j^*Y_j.
$$

\begin{lemma}
The family of vector fields $\underline{\widetilde{X}}_1,\ldots,\underline{\widetilde{X}}_m,Y_1,\ldots,Y_K$ satisfies Hörmander's bracket-generating condition \eqref{e:hormander}, it is regular at $(q,0)$, and the non-characteristic boundary condition (Assumption \ref{a:noncarac}) is verified on $\partial(\Omega\times \mathbb{T}^K)$.
 \end{lemma}
 \begin{proof}
In $V$ there holds $\underline{\widetilde{X}}_i=\widetilde{X}_i$. By Lemma \ref{l:desingularisation}, this implies that the Hörmander bracket-generating condition is satisfied in $V$. We notice that for any $i_1,\ldots,i_\ell\in \{1,\ldots,m\}$, 
$$[\underline{\widetilde{X}}_{i_1},[\underline{\widetilde{X}}_{i_2},[\ldots,\underline{\widetilde{X}}_{i_\ell}]]=[X_{i_1},[X_{i_2},[\ldots,X_{i_\ell}]] \qquad \mod Y_1,\ldots,Y_K.$$
Since the $Y_j$ do not vanish outside $V$, using that the vector fields $X_1,\ldots,X_m$ satisfy Hörmander's bracket-generating condition in $\Omega$, we obtain that the family of vector fields $\widetilde{X}_1,\ldots,\widetilde{X}_m,Y_1,\ldots,Y_K$ satisfies Hörmander's bracket-generating condition outside $V$.

The regularity of the family at $(q,0)$ follows from the fact that in $V$ there holds $\underline{\widetilde{X}}_i=\widetilde{X}_i$ and $Y_j=0$ for any $i\in\{1,\ldots,m\}$ and any $j\in\{1,\ldots,K\}$. By definition of the desingularized vector fields $\widetilde{X}_i$, they form an equiregular family, so in particular a regular family at $(q,0)$.\\
The non-characteristic boundary condition on $\partial(\Omega\times \T^K)$ follows from the non-characteristic boundary condition satisfied by the vector fields $X_i$ on $\partial \Omega$ (Assumption \ref{a:noncarac}) and the fact that $\underline{\widetilde{X}}_i=X_i$ on $\partial(\Omega\times \mathbb{T}^K)$.
 \end{proof}
 
 Let $\varphi_\lambda(x)$ be an eigenfunction of $-\Delta_\Omega$, with eigenvalue $\lambda$. We consider $\psi_\lambda:\Omega\times\mathbb{T}^K\rightarrow \R$ defined by $\psi_\lambda(x,y)=\varphi_\lambda(x)$. This is an eigenfunction of $-\widetilde{\Delta}$ with eigenvalue $\lambda$. 
 
 We apply the arguments of the beginning of Section \ref{s:endofth2} to $\widetilde{\Delta}$: they imply that there exist a neighborhood $\widetilde{V}\subset \Omega\times \T^K$ and a constant $c(\widetilde{V})>0$ independent of $\lambda$ such that 
\begin{equation}\label{e:Zpsi}
\forall \widetilde{q}\in \widetilde{V},\ \forall \e \geq c(\widetilde{V})\lambda^{-1/2}, \quad Z_{\psi_\lambda}\cap \widetilde{B}_\e(\widetilde{q})\neq \emptyset.
\end{equation}
These sR balls are computed with the vector fields defining $\widetilde{\Delta}$, and these vector fields coincide near $(q,0)$ with $\widetilde{X}_i$ thanks to \eqref{e:theunderlined}; hence it is equivalent to compute the sR balls with the vector fields $\widetilde{X}_i$ since we are considering small balls near $(q,0)$, with radius much smaller than $\delta$. By the projection property \eqref{e:projball}, since $Z_{\psi_\lambda}=Z_{\varphi_\lambda}\times \mathbb{T}^K$, we finally obtain that 
\begin{equation}\label{e:Zpsi}
\forall \widetilde{q}\in \widetilde{V},\ \forall \e \geq c(\widetilde{V})\lambda^{-1/2}, \quad Z_{\varphi_\lambda}\cap B_\e(\pi(\widetilde{q}))\neq \emptyset
\end{equation}
where $\pi:\Omega\times \mathbb{T}^K$ is the canonical projection. The constant involved in \eqref{e:Zpsi} is thus uniform in a neighborhood of $q$.
Since this uniformity is true in a neighborhood of any point $q\in M$ (either regular or singular), using the compactness of $\overline{\Omega}$ we obtain the result.

\begin{remark}
One could wonder why we do not simply consider in the proof the sub-Laplacian $-\sum_{i=1}^m\widetilde{X}_i^*\widetilde{X}_i$ on $U\times \R^K$ instead of $\widetilde{\Delta}$ (the adjoint being computed with respect to $\widetilde{\mu}=\mu\otimes \mathscr{L}_{\R^K}$, where $\mathscr{L}_{\R^K}$ is the Lebesgue measure on $\R^K$). In fact this does not work for our purposes, since the formula $\psi_\lambda(x,y)=\varphi_\lambda(x)$ does not define an $L^2(U\times\R^K)$ eigenfunction.
\end{remark}

\subsection{An example}\label{s:optimal}
In this section we illustrate Theorem \ref{t:density} with an example.\\ 
Fix $\alpha\in\N^*$ and consider the generalized Baouendi-Grushin sub-Laplacian $\Delta_{\rm BG}=\partial_x^2+x^{2\alpha}\partial_y^2$ on $(-1,1)_x\times \mathbb{T}_y$. For $k\in\Z$ we denote by $\psi_k$ a non-trivial element of the lowest energy eigenspace of the 1D operator $H_k=-\partial_x^2+k^2x^{2\alpha}$ on $(-1,1)_x$. The associated eigenvalue satisfies
 $$
c_\alpha|k|^{2/(\alpha+1)}\leq \mu_k\leq C_\alpha|k|^{2/(\alpha+1)}
 $$ 
 as $k\rightarrow +\infty$ for some constants $c_\alpha,C_\alpha>0$. Then $\Psi_k:(x,y)\mapsto\psi_k(x)\cos(ky)$ is an eigenfunction of $-\Delta_{\rm BG}$ with eigenvalue $\mu_k$. Its nodal set is
$$Z_{\Psi_k}=(-1,1)_x\times \left(\bigcup_{n\in\Z}\pi\frac{n}{k}\Z\right)_y$$
 i.e. it is a union of ``horizontal'' lines separated by $\pi/k$ (in Euclidean distance). \\
 Using the ball-box theorem \eqref{e:bboxthe}, the sR ball centered at a point $(0,y_0)$ on the singular line, and of radius $\e$, can be compared with ball-boxes of the form $[-C\e,C\e]_x\times [y_0-(C\e)^{\alpha+1},y_0+(C\e)^{\alpha+1}]_y$: the weights $w_i$ in this case are $w_1=1$ and $w_2=\alpha+1$. The sR ball is more squeezed in the $y$ direction due to the fact that brackets are needed to span this direction; we refer the reader to \cite[Section 3.1]{bellaiche1996tangent} and the picture in \cite[Section 3.3]{bellaiche1996tangent} for the case $\alpha=1$. Hence, the statement that any sR ball of radius $c/\sqrt{\mu_k}$ intersects the nodal set of $\Psi_k$ is true for $c$ large enough but false for $c$ small enough. This proves the sharpness of Theorem \ref{t:density}. 
 
 \section{Proof of Theorem \ref{t:yau}}\label{s:yau}
 We use the setting defined in Example \ref{exHeis} for $d=1$ (and we drop the indices: $x,y,z$ replace $x_1,y_1,z_1$): we consider $M =\Gamma\backslash \mathbf{H}_1$, with coordinates $x,y,z$, and endowed with the Lebesgue measure $\mu=dxdydz$. We denote by $\Omega$ the open subset of $M$ containing all points $(x,y,z)\in M$ such that $x\notin \sqrt{2\pi}(\Z+\frac12)$ (the factor $\frac12$ becomes clear in the proof). The boundary condition is then at $x=\pm \sqrt{\frac{\pi}{2}}\, (\text{mod} \sqrt{2 \pi})$.

We consider the vector fields $X=\partial_x$ and $Y=\partial_y-x\partial_z$ on $M$ (see Example \ref{exHeis}). The sub-Laplacian is $$\Delta=-X^*X-Y^*Y=X^2+Y^2.$$
Proposition \ref{p:eigpb} applies; we denote by $(\Delta_{\Omega}, \mathcal{D}(\Delta_{\Omega}))$ the domain of $\Delta$ acting on functions on $\Omega$ (with Dirichlet boundary conditions). We now proceed to explicit computations.

For $c\in \R$ we introduce the sets
\begin{equation*}
A_c=\left\{(x,y,z)\in \Omega\ | \ y=c\right\}, \qquad B_c=\left\{(x,y,z) \in \Omega\ | \ z=c\right\}.
\end{equation*}  
We make the following observations:
\begin{itemize}
\item All these hypersurfaces have Hausdorff dimension $3$ (see \cite[Section 0.6.C]{gromov}). 
\item All hypersurfaces $A_c$ have the same 3D Hausdorff measure, which we denote by $a$. This follows from the fact that for any $t\in\R$, the multiplication on the left by $(0,t,0)$ is an isometry, which sends $(x,y,z)$ to $(x,y+t,z)$, and thus $A_c$ to $A_{c+t}$ for any $c\in\R$. Since it is an isometry, it preserves balls and Hausdorff measures, and thus $A_c$ and $A_{c+t}$ have the same 3D Hausdorff measure, for any $c,t\in\R$.
\item All hypersurfaces $B_c$ have the same 3D Hausdorff measure, which we denote by $b$. This follows from the fact that for any $t\in\R$, the multiplication on the left by $(0,0,t)$ is an isometry, which sends $(x,y,z)$ to $(x,y,z+t)$, and thus $B_c$ to $B_{c+t}$ for any $c\in\R$. Since it is an isometry, it preserves balls and Hausdorff measures, and thus $B_c$ and $B_{c+t}$ have the same 3D Hausdorff measure, for any $c,t\in\R$.
\end{itemize}

We consider two sequences of eigenfunctions of $\Delta_\Omega$; it is not difficult to check that they are indeed well-defined on $\Omega$ and that they satisfy Dirichlet boundary conditions. 

We first consider eigenfunctions of the form $\varphi_{1,m}(x,y,z)=\sin(\sqrt{2\pi}x)\sin(\sqrt{2\pi}my)$ for $m\in\Z_{\geq 0}$, for which the corresponding eigenvalue is $\lambda_{1,m}=2\pi (1+m^2)$. The nodal set in $\Omega$\footnote{since $x\in (-\sqrt{\pi/2},\sqrt{\pi/2})$, $x=0$ is the only nodal set coming from $\sin(\sqrt{2\pi}x)$.} is the disjoint union of $\{x=0\}$ (whose Hausdorff measure is denoted by $a_0$) with the sets $A_c$ for $c=\frac{k}{m}\sqrt{\frac{\pi}{2}}$, $k\in \{-m+1,\ldots, m-1\}$. The Hausdorff measure of this nodal set is $a_0+(2m-1)a$ which is bounded above by  $Ca\sqrt{\lambda_{1,m}}$ for some $C>0$ independent of $m$.

Secondly, we consider for $m\in \Z_{\geq 0}$ the eigenfunction defined by extending to $\Omega$ by periodization with the group law \eqref{e:grouplaw} the function on the fundamental cell $(-\sqrt{\frac{\pi}{2}},\sqrt{\frac{\pi}{2}})\times [-\sqrt{\frac{\pi}{2}},\sqrt{\frac{\pi}{2}})\times [-\pi,\pi)$ given by $\varphi_{2,m}(x,y,z)=\psi_m(x)\sin(mz)$, where $\psi_m$ denotes a non-null element of the first eigenspace of the 1D operator $-d^2/dx^2+m^2x^2$ on $(-\sqrt{\frac{\pi}{2}},\sqrt{\frac{\pi}{2}})$ with Dirichlet boundary conditions. The first eigenvalue $\lambda_{2,m}$ of this harmonic oscillator, which is also the eigenvalue associated to $\varphi_{2,m}$, is $m+o(1)$ as $m\rightarrow +\infty$. Since $\psi_m$ does not vanish, the nodal set of $\varphi_{2,m}$ is the (disjoint) union of the sets $B_c$ for $c=k\pi/m$, $k\in \{-m+1,\ldots, m-1\}$. Its Hausdorff measure is $(2m-1)b$ which is bounded below by $C\lambda_{2,m} b$ for $m\neq 0$ and some $C>0$ independent of $m$.

This completes the proof of Theorem \ref{t:yau}.

\end{document}